\pdfoutput=1
\documentclass[paper=letter, fontsize=10pt,leqno]{scrartcl}
\usepackage[body={5.35in,7.5in}]{geometry} 
\usepackage[english]{babel}				
\usepackage{amsmath}				
\usepackage{amsfonts}
\usepackage{amsthm}
\usepackage{hyperref}
 \usepackage{verbatim}
\usepackage{graphicx}				
\usepackage{url}
\usepackage{amssymb}
\usepackage{bbm}
\usepackage{amscd}
\usepackage{calrsfs}
\usepackage[format=plain]{caption}
\usepackage{wrapfig}
 \usepackage{float}
\usepackage{sectsty}					
\allsectionsfont{\centering \normalfont\scshape}	
\usepackage{mathtools}
\DeclareCaptionStyle{ruled}{labelfont=bf,labelsep=space}
\usepackage[]{forest}

\newtheorem{theorem}{Theorem}

\newtheorem{definition}[theorem]{Definition}
\newtheorem{proposition}[theorem]{Proposition}

\newtheorem*{proposition*}{Proposition}

\title{
		\vspace{-1in} 	
		\normalfont \normalsize \textsc{} \\ [25pt]
		\huge  Nonholonomic billiards and bounded motion in cylinders
}

\author{\normalfont \large 
Christopher  Cox\footnote{\scriptsize Department of Mathematics and Statistics, University of Massachusetts Amherst, 710 N. Pleasant Street, Amherst, MA 01003},
\ Renato Feres\footnote{\scriptsize Department of Mathematics, Washington University, Campus Box 1146, St. Louis, MO 63130},
\ Zijie Hu\footnotemark[2]
}

\begin{document}

\maketitle

\begin{abstract}
\begin{center}
 Abstract \end{center}
{\small  
A widely used mathematical model for the bouncing motion of an ideally elastic ball\----referred to in  previous work by the first two authors and collaborators as a  {\em no-slip billiard} system\----exhibits some notable
dynamical behavior that is not well-understood. For example, under certain initial conditions,  the  axial component of the position of the  center of the ball moving inside a vertical solid cylinder under constant gravitational force does not  accelerate downward as might be expected but remains bounded. There is not  as yet, as far as we know, any  analytical study of
the bouncing ball dynamics, under gravity, in general cylinders (not necessarily having a circular  cross-section) in $\mathbb{R}^3$. In this paper, we propose an approach
 by comparing the no-slip system 
with a smooth approximation of it that we call {\em nonholonomic billiards}. It consists of a $4$-dimensional ball
rolling on the solid $3$-dimensional cylinder.  
 We first review earlier work on no-slip billiards and
their connection with nonholonomic (rolling) systems, explain how nonholonomic billiards approximate the no-slip kind (after work by the first two authors and
B. Zhao), 
and illustrate the relationship with a few numerical case studies that demonstrate the utility   of the soft (nonholonomic) system as a helpful tool for exploring the dynamics of  no-slip billiard systems.   
}
\end{abstract}

\section{Introduction}
The following geometric-mechanical problem is one of the motivations for  the present paper.  A perfectly elastic ball   is thrown against the inner surface of a cylindrical wall. The subsequent motion is under
 the influence of constant gravitational force directed downward along the cylinder axis. The ball then undergoes a sequence of bounces, with parabolic (falling) motion in between.
  We imagine the surface of the rigid, spherical ball as ideally rubbery, which causes the contact with the wall's inner surface to be perfectly nonslippery, in a sense to be shortly defined. 
  This no-slip property imposes at each collision a coupling of the ball's rotational and  linear momentum. 
  The cylinder need not have  circular cross-section; it encloses a region $\mathcal{R}=\mathcal{C}\times \mathbb{R}\subseteq \mathbb{R}^3$, where $\mathcal{C}\subseteq \mathbb{R}^2$. For example, when $\mathcal{C}$ is a strip in $\mathbb{R}^2$ bounded by two parallel lines, $\mathcal{R}$ is the region between two parallel, vertical walls.
  \begin{figure}[htbp]
\begin{center}
\includegraphics[width=2.5in]{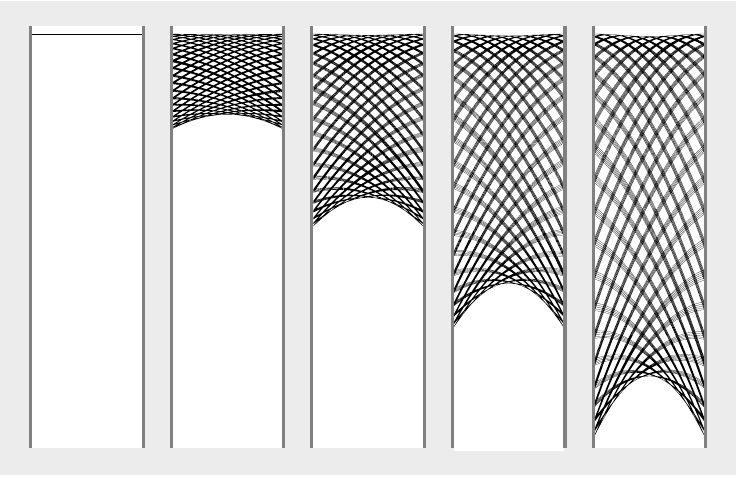}\ \ 
\caption{{\small  No-slip disc bouncing between two parallel vertical lines, under gravity, with increasing values of the acceleration-due-to-gravity parameter $g$ increasing from left to right. On the far left, $g=0$ and the trajectory is periodic of period $2$. }}
\label{plates}
\end{center}
\end{figure}    

  The question we ask is then:
     {what can be said about the ball's vertical motion?} In particular, can the  motion remain bounded so that  the ball does not fall below a certain height?
     In this section we review a few pertinent results from a couple of   papers by the first two authors together with B. Zhao \cite{CFZ21}, and T. Chumley and S. Cook \cite{CCCF20}. It turns out that trajectories of the no-slip billiard system under gravity are bounded in the case of parallel vertical walls (in any dimension); see
    Figure \ref{plates}, which shows trajectories of the center of the ball in dimension $2$ under increasing values of the acceleration due to gravity.  
     
    \begin{figure}[htbp]
\begin{center}
\includegraphics[width=5.0in]{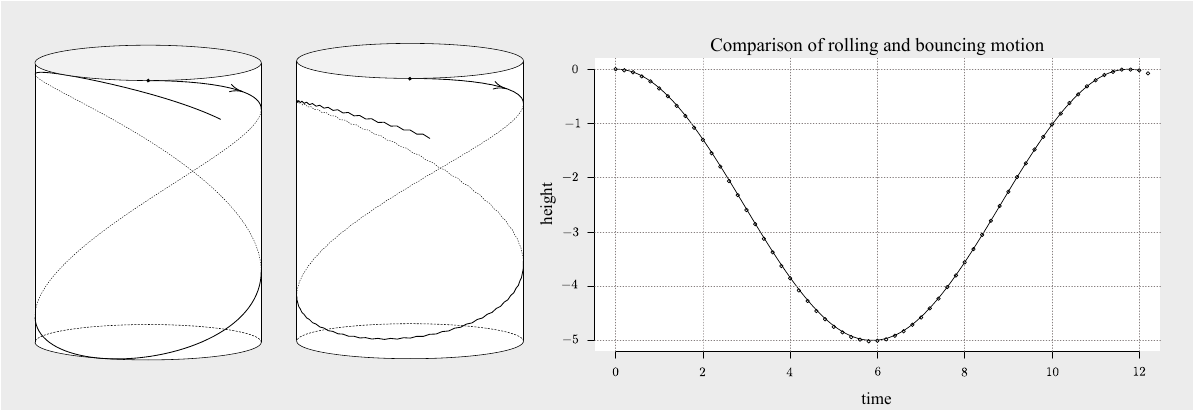}\ \ 
\caption{{\small Left:  segment of trajectory of the center of a ball  rolling on the inner surface of  a cylinder, under gravity, showing harmonic vertical motion. Center:   trajectory of a no-slip billiard ball, with very short bouncing steps, not satisfying the initial rolling impact assumption. Right: under that assumption, the no-slip billiard ball closely tracks the rolling motion. }}
\label{cylinder}
\end{center}
\end{figure}

The motion is also bounded in $3$-dimensional  circular cylinders   under a natural  assumption on   initial conditions  that we call {\em rolling impact}.  This case is particularly interesting as it highlights the similarities with the classical and well-known mechanical  system consisting of a ball rolling without  slipping on the inner wall of  a cylinder, under gravity, for which the vertical component of the motion oscillates harmonically. See \cite{CCCF20} and  the left image in the below Figure \ref{cylinder}.  The  initial condition  on the linear and angular velocities  ensuring  bounded motion is this: the velocity of  the point on the ball in contact  with the cylindrical  surface at  the time of first collision should have zero component along the tangent line to the circular cross section. In other words, this initial velocity (which, we stress, is not the center velocity but the velocity of a boundary point on the ball) should lie in the vertical plane containing the cylinder axis
and   the point of contact. We say in this case that the initial  (linear and angular) velocities satisfy a {\em rolling impact} condition.

The vertical component of the center of the  ball bouncing inside  the circular cylinder is  shown as a function of time by the graph on the right-hand side of 
 Figure \ref{cylinder}. There it is assumed that the first collision satisfies  rolling impact.   This generates a sequence of closely spaced bounces (think of a skipping stone) that
 precisely tracks the rolling motion.

    \begin{figure}[htbp]
\begin{center}
\includegraphics[width=4 in]{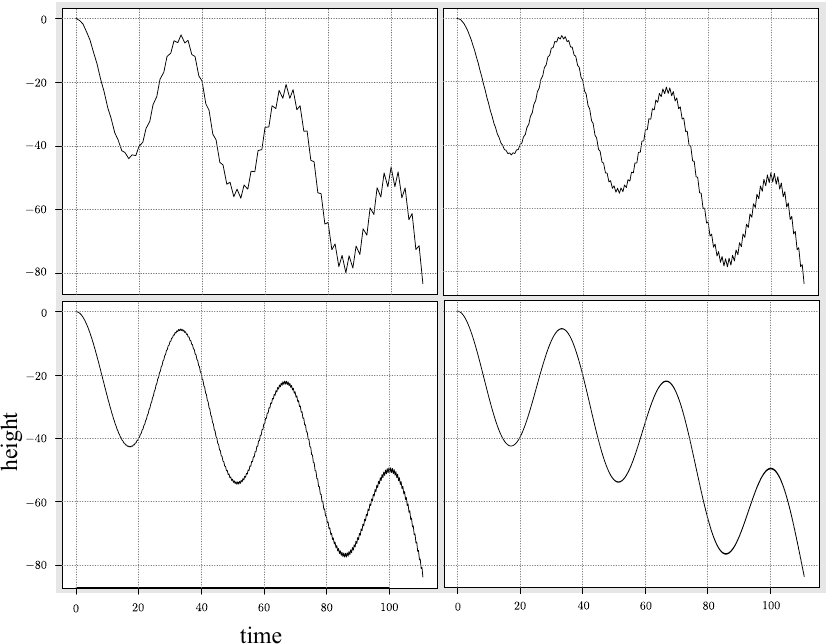}\ \ 
\caption{\small{Height of the ball's center as a  function of time when the rolling first impact assumption is not satisfied. The sequence of graphs (NW $\rightarrow$ NE $\rightarrow$ SW $\rightarrow$ SE) was generated with decreasing   lengths of the skipping motion steps. An apparent energy dissipation into the  zig-zag motion at smaller and smaller scales produces accelerated falling and  decreasing amplitude of oscillation.}} 
\label{falling_down}
\end{center}
\end{figure}   
    
    Boundedness of trajectories was noted in the two dimensional strip with gravity in 
  \cite{Hefner2004}  and then later proved rigorously  for more general cases in   \cite{CCCF20}, but neither argument is as illuminating as may be wished.  For example, the proof does not 
 relate the bouncing  motion to  rolling   and so  does not explain why  the former can so closely shadow the latter as shown in Figure \ref{cylinder}. 
In particular, one can expect  the skipping motion illustrated by the graph of Figure \ref{cylinder} to converge to a solution to the rolling differential equations (to be given later). Obtaining this limit is so far, to us,   a challenging problem. Part of the subtlety  is already apparent in the middle image of Figure \ref{cylinder}. When the rolling impact condition is not satisfied, the bouncing motion  with short bouncing steps exhibits  a zig-zag small scale component 
that increases in amplitude with time. In the limit of bouncing steps approaching $0$, it is natural to conjecture that some sort of rolling trajectory is obtained which, however, incorporates   energy dissipation through this infinitesimal zig-zag   motion, causing the ball's descent to accelerate and the amplitude of oscillation to decrease, as illustrated in Figure \ref{falling_down}. Obtaining the equations for such  dissipative rolling motion is an interesting problem that is not addressed here.
 
\begin{figure}[htbp]
\begin{center}
\includegraphics[width=3.5 in]{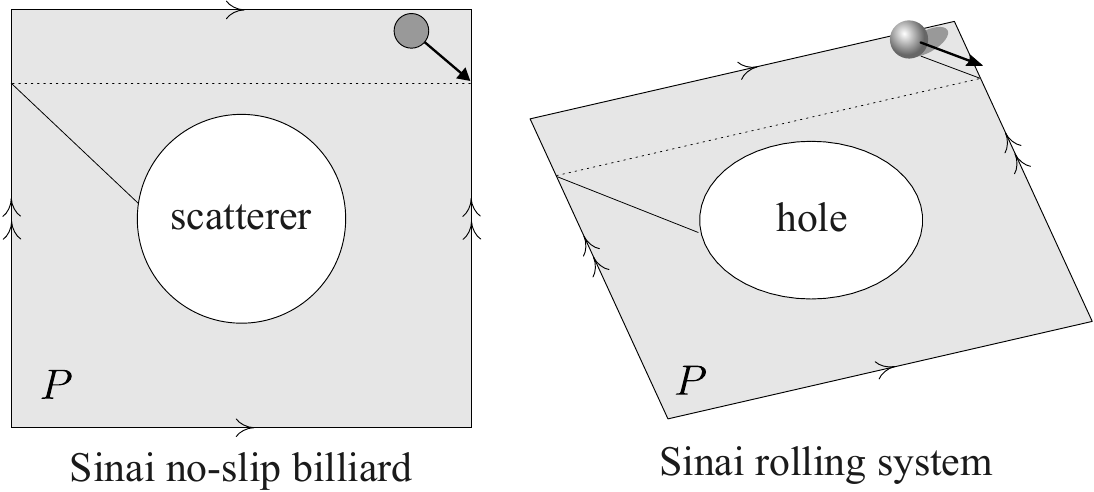}\ \ 
\caption{\small{Left: a two-dimensional no-slip billiard system on a Sinai billiard table. Right: the associated $3$-dimensional rolling system. With appropriately adjusted moment of inertia parameters for the disc and the ball, the rolling-around-the-edge of the circular hole,  on the right, converges to the no-slip billiard reflection off the circular scatterer, on the left,  as the radius of the  ball approaches $0$. This suggests that the two systems have comparable dynamical properties, and that the rolling system, governed by differential equations, can help  us to understand the other\----impact driven\----system. The one  on the left may be regarded  as the two-dimensional shadow of that on the right, while the one on the right defines a soft-collisions version of the other.}} 
\label{Sinai_compare}
\end{center}
\end{figure} 
     
A promising  approach to a detailed analysis of  these no-slip billiard systems is suggested by \cite{CFZ21}, which establishes a much more direct connection between no-slip billiards and nonholonomic  systems. Stated somewhat vaguely for now,  a no-slip billiard system in an $n$-dimensional submanifold $\mathcal{R}\subseteq\mathbb{R}^n$ with boundary can be approximated, for small values of the ball radius, by the  rolling of an $n+1$-dimensional ball on $\mathcal{R}\subseteq \mathbb{R}^{n+1}$, with appropriately set moment of inertia parameter. See Figure \ref{Sinai_compare}. (Away from the boundary edges and
in the absence of gravity,   the motion of the center of the rolling ball  is uniform and rectilinear.)  In this way, the study of a billiard-like system driven by discrete impulsive forces can, in principle, be approached via the analysis of a nonholonomic system governed by differential equations, at the cost of passing from dimension $n$ to dimension $n+1$. We call the latter class of systems  {\em nonholonomic billiards}.  This
 term has been previously    used in \cite{BKM}, which was the inspiration for \cite{CFZ21}. In the former paper, however,  it has  a different definition in  a much more
 restricted setting. The two types of nonholonomic  billiards are, nevertheless,  sufficiently close conceptually   that using   the same name seems  to us appropriate.

Our main goal in this paper, in addition to reviewing a few of the results mentioned above, is to explore, by a combination of analytic and numerical methods, some examples of such nonholonomic billiard systems in cylinders under gravity and indicate the similarities with their no-slip counterparts.

\section{No-slip billiard systems}
\subsection{Generalities}\label{generalities}
A few definitions are needed to make precise the above description of the {\em no-slip} bouncing ball,  referred to here as a  {\em no-slip billiard system}. 
Let the ball of radius $r$ centered at the origin of $\mathbb{R}^n$ be  $B:=B_0(r)$ and define a mass distribution on $B$  by a rotationally invariant finite Borel measure $\mu$ on $B$ of total mass $m$. The mass-normalized  matrix $L=(\ell_{ij})$ of second moments of $\mu$ is defined by
$$\ell_{ij}:=\frac1m \int_{B} x_i x_j\ d\mu(x),$$
where $x=(x_1, \dots, x_n)$ are the standard Cartesian coordinates in $\mathbb{R}^n$. Rotational symmetry implies that $L$ is a scalar matrix, $L=\lambda I$, where
$I$ is the identity matrix in dimension $n$. It will be useful to express the  moment of inertia parameter $\lambda$ in a couple of different forms. We first define the dimensionless parameter
$$\gamma:=\sqrt{2\lambda}/r. $$ 
For example, $\gamma=\sqrt{\frac{2}{n+2}}$ for a uniform mass distribution in dimension $n$. Let   $\mathcal{R}_0$ be a region in $\mathbb{R}^n$
 and define the {\em billiard domain} $\mathcal{R}$ as the closure of  the set of $x\in \mathbb{R}^n$ for which the open ball $B_x(r)$ of center $x$ and radius $r$
 is contained in $\mathcal{R}_0$. We assume that $\mathcal{R}_0$ is sufficiently regular and $r$ sufficiently small   that $\mathcal{R}$ has  piecewise smooth boundary, possibly with corners. On a regular point $a$ on the boundary of $\mathcal{R}$ (that is, a point at which the tangent space to the boundary of $\mathcal{R}$ exists) define the unit normal vector $\nu(a)=\nu_a$ pointing towards the interior of $\mathcal{R}$.  Denote by $SE(n)$ the special Euclidean group of positive isometries of $\mathbb{R}^n$, whose elements are pairs $(A,a)\in SO(n)\times \mathbb{R}^n$ representing the affine transformation
 $x\mapsto Ax+a$. The configuration manifold of the ball-in-$\mathcal{R}_0$ system is $M\subseteq SE(n)$ consisting of all $(A,a)\in SE(n)$ such that $a\in \mathcal{R}$. 
 Thus $M$ is a submanifold of $SE(n)$ with boundary and possibly corners, trivially fibered over $\mathcal{R}$, with typical fiber given by 
  the rotation group $SO(n)$. Elements in the Lie algebra of $SE(n)$ will be written $(U,u)$ where $U\in \mathfrak{so}(n)$\----the {\em angular velocity}\----is a skew-symmetric $n\times n$ matrix and $u\in \mathbb{R}^n$ is the velocity of the center of the ball. 
 
 From the  mass distribution measure $\mu$  we obtain the kinetic energy Riemannian metric in $M$ as follows. Tangent vectors  $\xi, \eta$  in  $T_{q}M$ at  configuration $q=(A,a)$ may be written
 as $\xi=(U_\xi A, u_\xi)$
with $(U_\xi, u_\xi)\in \mathfrak{so}(n)\times \mathbb{R}^n$.
The kinetic energy of the system in configuration $q$ and kinetic state $\xi$ is obtained by integrating over $B$ the kinetic energy  $$\frac12 |V_\xi(x)|^2\, d\mu(x)$$ of a mass element $d\mu(x)$,  where $V_\xi(x)\in \mathbb{R}^n$ is the velocity vector of the material point $x$ in the   kinetic state $\xi$.
The  inner product on $T_{q}M$ associated to the kinetic energy  quadratic form  is easily obtained:
$$\langle \xi,\zeta\rangle_q := m\left\{\frac{(r\gamma)^2}{2}\text{Tr}\left(U_\xi U_\zeta^\intercal\right) + u_\xi\cdot u_\zeta\right\}. $$

The tangent bundle to the boundary of $M$ contains a  vector subbundle $\mathfrak{R}$, which we call the {\em rolling} subbundle, needed for the definition of  the no-slip condition.   At a regular boundary point $q=(A,a)$ of $M$ (so that $a$ lies in the regular boundary of $\mathcal{R}$) the rolling subspace of $T_qM$ is
$$ \mathfrak{R}_q:=\left\{(UA, u )\in T_qM: u=rU\nu_a\right\}.$$
Note that  the point on the boundary of the ball    in contact with the boundary of $\mathcal{R}_0$ at a collision configuration has zero velocity in the kinetic state $(UA,u)\in \mathfrak{R}_q$. This is implied by the equation $u=rU\nu_a$. 
\subsection{The no-slip collision map}
The ball's kinetic state  immediately after a collision,  $C_q \xi$, is a function of  the state $\xi$  immediately before. This {\em collision map} $C_q:T_qM\rightarrow T_qM$ is assumed to have  the following properties: (1) it is a linear, (2) preserves kinetic energy, (3) satisfies time reversibility, (4) it restricts to the identity on
 the rolling space $\mathfrak{R}_q$. This last assumption  is justified by noting  that, at the moment of collision, 
the point on the moving ball in contact with the boundary of $\mathcal{R}_0$ has zero velocity for a kinetic state  in $\mathfrak{R}_q$, and thus no change 
 in  linear or angular momentum should occur. Such $\xi\in \mathfrak{R}_q$ physically corresponds to a rolling and grazing collision. We refer to \cite{CF}, by the first two authors, for
more details on the classification of such collision maps. It is interesting to observe, as noted in \cite{CF}, that condition (4)  follows from the physically natural requirement that
  impulse forces causing the ball to bounce off the wall hypersurface  must act only upon the point of contact.  (In principle, one may conceive of more general 
impulse force fields  with an extended range of action, such as would happen when shape deformation during  a collision event is also modeled.)  
These four  assumptions characterize $C_q$ as an orthogonal linear involution that restricts to the identity map on the rolling subspace. 

 If we further assume isotropic  ball surface properties (effectively, the map $C_q=C_a$ does not depend on  $A$), then it is not difficult  to
conclude (see \cite{CF}) that there are only two possibilities for $C_q$, determined by whether this map's restriction to $\mathfrak{R}_q^\perp$\----the orthogonal complement to the rolling subspace\----is the identity or its negative: in the first case, which we call {\em slippery} or {\em standard} collision map (the term {\em smooth} is
widely used in the physics and engineering  literature), the velocity of the center of the ball reflects specularly and
angular momentum is unchanged; ignoring angular momentum and tracking the motion by the position of the ball's center, one obtains  ordinary billiard dynamics.   The second possibility, for which $C_q|_{\mathfrak{R}_q^\perp}$ is the negative of the identity map, corresponds to
what we call {\em no-slip} collision. 

 The explicit form of $C_q=C_a$ acting on $\xi=(UA, u)\cong (U,u)\in \mathfrak{se}(n)$ (the latter being the Lie algebra of $SE(n)$) is
 \begin{equation}\label{collision_map}C_a(U,u)=\left(U+ \frac{s_\beta}{\gamma r} \nu_a\wedge\left[ u -rU\nu_a\right],c_\beta u -\frac{s_\beta}{\gamma}(u\cdot \nu_a)\nu_a + s_\beta \gamma r U\nu_a\right),\end{equation}
 where   
 $$c_\beta:=\cos\beta :=\frac{1-\gamma^2}{1+\gamma^2}, \ \ s_\beta:=\sin\beta :=\frac{2\gamma}{1+\gamma^2}, $$
 $u \cdot v$ is the ordinary dot product,
 and $\wedge: \mathbb{R}^n\times \mathbb{R}^n\rightarrow \mathfrak{so}(n)$ is the generalized cross product of vectors: if $u, v\in \mathbb{R}^n$ then
 $u\wedge v$ is the skew-symmetric matrix such that
 $$(u\wedge v)w:= (u\cdot  w)v -(v\cdot w)u. $$

 The no-slip  assumption is typically used as  model for   energy preserving impacts involving rigid  spherical bodies when coupling of linear and angular momentum  is expected at collisions. Examples are given in \cite{ gutkin,CFII,CFZ,Garwin,gualtieri,MLL,W}. 

A more revealing  form  of Equation (\ref{collision_map}) is obtained with the following  minor change in notation. Let $V_a$
be the tangent space to the boundary of $\mathcal{R}$ at a regular boundary point $a$ and denote by  $\Pi_a$  the orthogonal projection 
from $\mathbb{R}^n$ to $V_a$.  Then
$$T_a\mathcal{R}=V_a\oplus R\nu_a$$
and,
as shown in \cite{CFZ21}, any $U\in \mathfrak{so}(n)$ has the decomposition
$$U=\Pi_aU\Pi_a +\nu_a \wedge U\nu_a,$$
where $\Pi_a U \Pi_a\in \mathfrak{so}(V_a)$ is    an infinitesimal rotation on $V_a :=T_a(\partial \mathcal{R})$. Now
define
 $$\overline{u}:=\Pi_a u,\ \   S:=\gamma r U, \ \   \overline{S}:=\Pi_a S\Pi_a,  \ \ 
 W:= W_S:=S\nu_a.$$
 Both $\overline{u}$ and   $W$ are tangent vectors in $V_a$.
 Let us use the same symbol $C_a$ to represent the collision map in the new variables $(S,u)$. Then
 Equation (\ref{collision_map}) becomes
\begin{equation}\label{collision_S}C_a(S,u)= \left(\overline{S}, 0\right) +(0, -u\cdot\nu_a \nu_a) + \left(\nu_a\wedge\left(s_\beta \overline{u} - c_\beta W\right),  c_\beta \overline{u}+s_\beta W\right). \end{equation}
 Separating components, note that
  \begin{equation}\label{collision_map_bis}
 \overline{S} \mapsto   \overline{S}, \ \ \nu_a\mapsto -\nu_a, \ \ \left(\begin{array}{c}\overline{u} \\W\end{array}\right) \mapsto \left(\begin{array}{cc}c_\beta  & \ \ s_\beta  \\s_\beta  & -c_\beta \end{array}\right)\left(\begin{array}{c}\overline{u} \\W\end{array}\right). 
  \end{equation}
  For simplicity, we have omitted  the $(n-1)\times (n-1)$ identity matrix (more precisely, the identity map on $V_a$) multiplying each $s_\beta$, $c_\beta$ in the above block matrix. Since $c_\beta^2+s_\beta^2=1$, the map
  $$ \left(\begin{array}{c}\overline{u} \\W\end{array}\right) \mapsto \left(\begin{array}{cc}c_\beta  & \ \ s_\beta  \\s_\beta  & -c_\beta \end{array}\right)\left(\begin{array}{c}\overline{u} \\W\end{array}\right)=
 \left(\begin{array}{cc}c_\beta  & -s_\beta  \\s_\beta  & \ \ c_\beta \end{array}\right)\left(\begin{array}{c}\overline{u} \\ -W\end{array}\right)  $$
is, as expected, an orthogonal transformation with negative determinant. 
 The identity on the right shows that
the transformation amounts to a sign change of $W$ composed with a rotation by the characteristic angle $\beta$ (a parameter serving as proxy  for the moment of inertia)
that mixes up (and thus couples) linear velocity and certain components of the angular velocity matrix. When the moment of inertia parameter $\gamma$ is zero,
the characteristic angle $\beta$ is also zero and the block matrix reduces to the identity. In this case,  
$$\left(\begin{array}{c}\overline{u} \\W\end{array}\right)\mapsto \left(\begin{array}{c}\overline{u} \\ -W\end{array}\right), $$
so that  $\overline{u}$ and $W$ are fully decoupled.  Later we will return   to this collision map  and compare it  with a rolling-over-the-edge map for nonholonomic billiards.
A manifestation of time reversibility is  the involutive property
$$\left(\begin{array}{cc}c_\beta  & \ \ s_\beta  \\s_\beta  & -c_\beta \end{array}\right)^2=I. $$
It is useful to register here that, using the $\overline{S}$ and $W$ definitions, the kinetic energy Riemannian metric takes the form
\begin{equation}\label{SWu} \langle \xi,\zeta\rangle_q:=m\left\{\frac12 \text{Tr}\left(\overline{S}_\xi \overline{S}_\zeta^\intercal\right)+W_\xi\cdot W_\zeta +u_\xi\cdot u_\zeta\right\}.
\end{equation}

\subsection{Dimensions $2$ and $3$}
For two-dimensional regions ($n=2$), $\overline{S}=0$ and the inner product (\ref{SWu}) reduces to standard dot product in $\mathbb{R}^3$, in which the third variable is  (the one-dimensional)  $W$.  In this case, $S=sJ$, where $J$ is positive (counterclockwise) rotation by $\pi/2$ in $\mathbb{R}^2$ and at a boundary point $a$ of the
planar billiard region $\mathcal{R}$, we have $W=S\nu = s \tau_a$, where $\tau$ is the unit tangent vector to the boundary curve of $\mathcal{R}$ oriented counterclockwise.

\begin{figure}[htbp]
\begin{center}
\includegraphics[width=4.0 in]{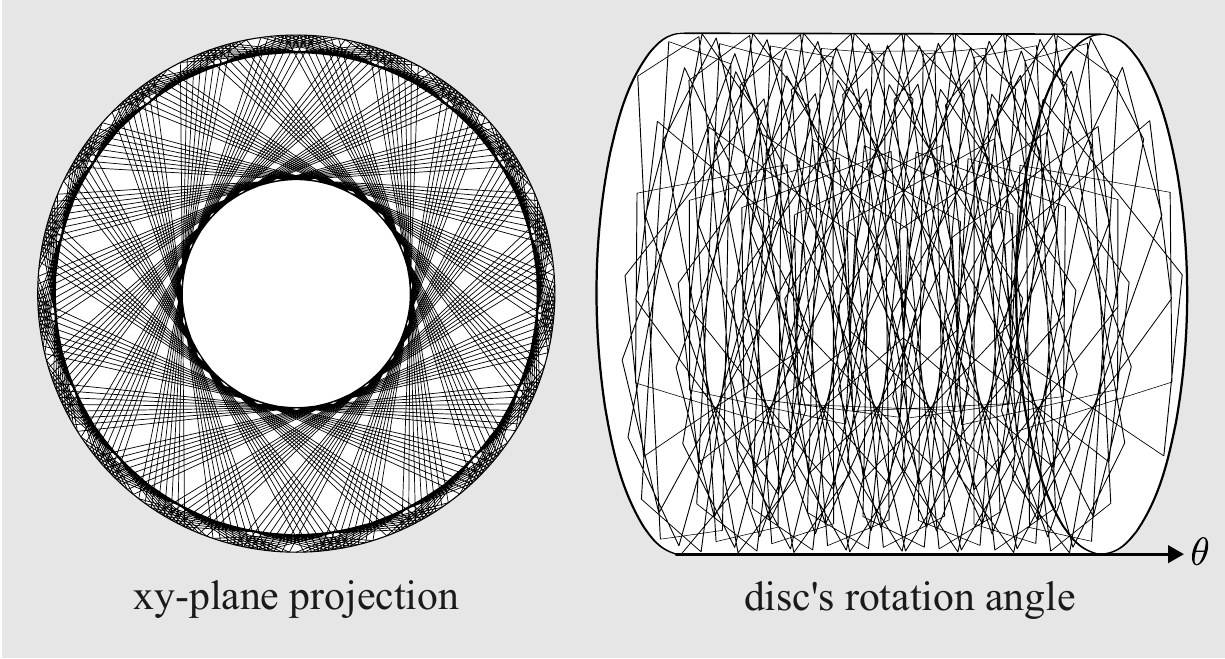}\ \ 
\caption{\small{Trajectory of a no-slip billiard system on a disc.  Left: trajectory of the moving disc's center, showing the characteristic double caustic. Right: $3$-dimensional configuration space trajectory, including the moving  disc's rotation angle. }} 
\label{noslip_disc}
\end{center}
\end{figure}

The kinetic state of the billiard disc is specified by $(\hat{u}, \overline{u}, s)$, where $\hat{u}=u\cdot \nu_a$ and $\overline{u}=u\cdot\tau_a$, and the collision map becomes
$$C_a: \left(\begin{array}{c}\hat{u}  \\\overline{u} \\ s\end{array}\right) \mapsto \left(\begin{array}{rcc}-1 & 0 & 0 \\0 & c_\beta & s_\beta \\ 0 & s_\beta  &\!\! -c_\beta\end{array}\right) \left(\begin{array}{c}\hat{u}  \\\overline{u} \\s \end{array}\right).$$
Figure \ref{noslip_disc} shows a segment of trajectory of a no-slip billiard system on a disc shaped table. The actual trajectory is in the  $3$-dimensional solid torus $D\times S^1$,
where the disc $D$ contains the positions of the center of the moving particle (also a disc) and $S^1$ the orientations in space. (The figure shows an orbit segment in the covering space $D\times \mathbb{R}$ rather than $D\times S^1$.)
The projection of the center of the moving disc to the circular table (left of Figure \ref{noslip_disc}) shows a characteristic double caustic, an easy to understand feature based solely on  time reversibility and symmetry of the billiard domain. When the moment of inertia is $0$ ($\gamma=0$), one obtains an ordinary billiard trajectory with a single caustic circle. 

\begin{figure}[htbp]
\begin{center}
\includegraphics[width=2.5 in]{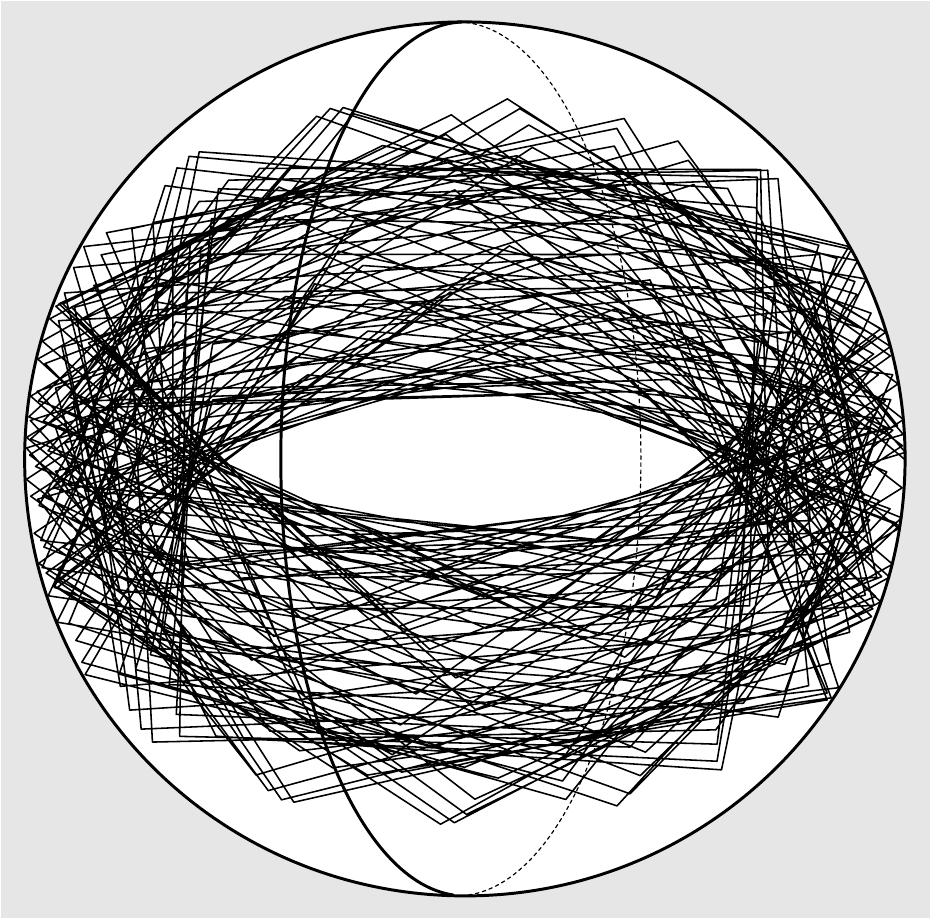}\ \ 
\caption{\small{A segment of trajectory of the center of a no-slip ball in a sphere.}} 
\label{sphere}
\end{center}
\end{figure} 

For $n=3$, let $J_a$ be the operator of positive rotation by $\pi/2$ on the $2$-dimensional $V_a$ (the tangent space to the boundary of $\mathcal{R}$ at $a$.) Here {\em positive} is defined by the outward pointing $\nu_a$ (and the right-hand rule of vector calculus.) Then $\overline{S}=\overline{s}J_a$, $\overline{s}\in \mathbb{R}$ may be called
the {\em tangential spin} at the point of contact with the boundary. In this case, $(\overline{s}, W)$ may be used as proxies for $S$ and
$$ \langle \xi,\zeta\rangle_q:=m\left(\overline{s}_\xi \overline{s}_\zeta+W_\xi\cdot W_\zeta +u_\xi\cdot u_\zeta\right),$$
while the no-slip collision map takes the form
$$C_a:\left(\begin{array}{c}\overline{s} \\ \hat{u} \\\overline{u} \\W \end{array}\right) \mapsto \left(\begin{array}{crcc}1 & 0 & 0 & 0 \\0 & -1 & 0 & 0 \\0 & 0 & c_\beta & s_\beta \\0 & 0 & s_\beta &\!\! -c_\beta\end{array}\right) \left(\begin{array}{c}\overline{s} \\ \hat{u}  \\\overline{u} \\W\end{array}\right). $$
This is a linear map in $SO(6)$. As  before, $\hat{u}:=u\cdot \nu_a.$ The invariance of $\overline{s}$ makes physical sense: a change in $\overline{s}$ would require an infinite impulsive  torque tangent to the wall,  acting on the ball at the single point of contact. It is also not surprising that the normal component of the linear velocity
changes sign. The essence of the no-slip collision map lies in the coupling  of linear and angular velocities controlled by the parameter $\gamma$ that defines $c_\beta, s_\beta$.

As an example, Figure \ref{sphere} shows a segment of trajectory of the center of a no-slip ball bouncing inside a sphere. The figure shows a characteristic spherical annulus containing the trajectory. 
It is not clear how to explain all the (numerically) observed features of this example based on symmetry considerations. For example, it is an interesting exercise   to determine the axis and width of this annular envelope as a function of a trajectory's initial conditions.

\subsection{No-slip billiards in general cylinders}
Let $\mathcal{R}=\overline{\mathcal{R}}\times \mathbb{R}$ be an $(n+1)$-dimensional domain with $\overline{\mathcal{R}}\subseteq\mathbb{R}^n$ a billiard domain in 
dimension $n$. We call such region a {\em cylinder}. 
 We refer to  $e:=(0,\dots, 0, 1)$  as the {\em cylinder axis}  unit vector and to $\overline{\mathcal{R}}$ as the {\em cross-section}. Let $\nu$ be the outward pointing unit normal vector field on the boundary of $\mathcal{R}$, which is also a unit normal vector field on the boundary of $\overline{\mathcal{R}}$.
It turns out that the orthogonal projection of trajectories of the no-slip billiard ball in $\mathcal{R}$ to the cross-section are trajectories of the no-slip
$n$-dimensional (projected) ball on   $\overline{\mathcal{R}}$. The two balls in different dimensions are assumed to have the same moment of inertia parameter $\gamma$ or $\beta$.

\begin{theorem}\label{theorem_project}
Let $\Pi$ be  the orthogonal projection from $\mathbb{R}^{n+1}$ to $\mathbb{R}^n$. Then $\Pi$ maps trajectories of the no-slip billiard system in a cylinder of  dimension $n+1$ to
trajectories of the corresponding cross-sectional no-slip billiard system in dimension $n$, assuming the moment of inertia parameters of the two systems are the same.
This holds even when the $(n+1)$-dimensional ball is subjected to a constant force parallel to the axis of the cylinder.
\end{theorem}

Before proving Theorem \ref{theorem_project}, we need  a few definitions.
At any regular  $a$ on the boundary of $\mathcal{R}$, let $\Pi^\perp_{\nu_a}$ be the orthogonal projection to the hyperplane $\nu_a^\perp$, $\Pi^\perp_e$ be the
orthogonal projection to the hyperplane $e^\perp$, and $\Pi_a$ be the orthogonal projection to the codimension-$2$ subspace $(\text{span}\{\nu_a, e\})^\perp$.
Let $\mathfrak{se}\left(e_a^\perp\right)$ be the Lie algebra of the Euclidean group on the hyperplane orthogonal to $e$ at $a$ and $\mathfrak{se}(n)$ be the
Lie algebra of the Euclidean group on $\mathbb{R}^n\cong \mathbb{R}^n\times\{0\}$, regarded as the subspace of $\mathbb{R}^{n+1}$. These two Lie algebras are naturally isomorphic.   

Theorem \ref{theorem_project} is now a corollary of the following three propositions.
\begin{proposition}
Define    the $(n+1)$-dimensional space 
\begin{equation}\label{direct}\left(e\wedge e^\perp_a, \mathbb{R}e\right):=\left\{(e\wedge v, \lambda e)\in \mathfrak{se}(n+1): v\cdot e=0 \text{ and } \lambda\in \mathbb{R}\right\}.\end{equation} Then
$$\mathfrak{se}(n+1)=\mathfrak{se}\left(e_a^\perp\right)\oplus \left(e\wedge e^\perp, \mathbb{R} e\right)$$
 is an orthogonal direct sum of vector spaces. (The second summand is not a Lie algebra.)
\end{proposition}
\begin{proof} Orthogonality relative to the inner product
$$\langle (S_1, u_1), (S_2, u_2)\rangle := m\left\{\frac{1}{2}\text{Tr}\left(S_1 S_2^\intercal\right) + u_1\cdot u_2\right\} $$
is readily checked. Equality then follows from counting dimensions:  
$$ \frac{(n+1)(n+2)}{2} = \frac{n(n+1)}{2} + n+1$$
and $\text{dim }\mathfrak{se}(n)=\frac{n(n+1)}{2}$.
\end{proof}

\begin{proposition}
The no-slip collision map $C_a$ respects the orthogonal decomposition (\ref{direct}).
\end{proposition}
\begin{proof}
Since  $C_a$ is orthogonal, 
it suffices to show that $C_a$ maps $\left(e\wedge e^\perp, \mathbb{R} e\right)$ to itself. Now
$$C_a(e\wedge v, \lambda e)=\left(e\wedge v', \lambda' e\right)$$
where
$$v' = \Pi_{a}v - \left(s_\beta \lambda + c_\beta \nu_a\cdot v\right)\nu_a, \ \ \lambda'=c_\beta\lambda -s_\beta \nu_a\cdot v. $$
 This is obtained from a direct application of $C_a$ in the form of Equation (\ref{collision_S}).
\end{proof}

\begin{proposition}
The projection map  $\Pi_e^\perp$ intertwines the collision map $C_a$ on $\mathfrak{se}(n+1)$ associated to the no-slip billiard system on $\mathcal{R}$ and the collision map $C_{\overline{a}}$ on $\mathfrak{se}(n)$ associated to the no-slip billiard system on $\overline{\mathcal{R}}$,   where
$\overline{a}:=\Pi_e^\perp a$. Both maps $C_a$ and $C_{\overline{a}}$ are assumed to have the same moment of inertia parameter $\gamma$ (or $\beta$).
\end{proposition}
 \begin{proof}
We can verify that  $\Pi_e^\perp\circ  C_a = C_{\overline{a}}\circ \Pi_e^\perp$ holds by  a straightforward application of the explicit form of the collision maps as
given in Equation (\ref{collision_S}). One should keep in mind that $\Pi_e^\perp$ maps $(S,u)$ to $\left(\Pi_e^\perp S\Pi_e^\perp, \Pi_e^\perp u\right).$
 \end{proof}

Note that  an external force  cannot affect the discontinuous change in velocities at a collision event. Physically,   impact forces are very intense over a very short time interval,  during which
ordinary forces will have negligible contribution to  changing the particle's linear and angular  momentum. Additionally,   a force parallel to the axis of the cylinder will not affect the projection of trajectories to cross-sectional hyperplanes.

\subsection{A remark concerning  symmetries}
It is useful to note that Euclidean isometries (not necessarily orientation preserving) mapping the boundary of $\mathcal{R}$ to itself are symmetries
of the no-slip system in the sense of the following proposition. Notations are as follows.
Suppose $f\in E(n)$ is an Euclidean isometry (not necessarily positive) mapping the regular boundary point $a$ in $\mathcal{R}$ to another such point $f(a)$ for which
$F_a \nu_a=\nu_{f(a)}$, where $F_a:=df_a$ is the differential of $f$ at $a$. It follows that $F_a$ maps $V_a$ to $V_{f(a)}$ (recall that $V_a$ denotes the tangent space at $a$ to the boundary of $\mathcal{R}$.) It is easily checked that $$F_a\Pi_a =\Pi_{f(a)} F_a, \ \ F_a\overline{u}=\overline{F_a u},\ \ F_a (u\wedge v) F_a^{-1} =F_a u \wedge F_a v, \ \ F_a W_S = W_{F_a SF_a^{-1}}.$$
Writing, for simplicity, $f_au:=F_au$ and  $f_a S:=F_a S F_a^{-1}$, we obtain the transformations
$$\overline{u}\mapsto f_a\overline{u}= \overline{f_a u}, \ \ \overline{S}\mapsto f_a\overline{S}= \overline{f_a S},  \ \  W_S\mapsto f_a W=W_{f_a S}, \ \ \left(\overline{u}, W\right)\mapsto \left(f_a\overline{u}, f_a W\right).$$
From these and (\ref{collision_map_bis}) we obtain the equivariant property of the collision map:
$$f_a\circ C_a = C_{f(a)}\circ f_a. $$
This holds, in particular, whenever the Euclidean isometry $f\in E(n)$ restricts to a self-map  of (the boundary of) $\mathcal{R}$. 
We say in this case that $f$ is a {\em symmetry} of the no-slip billiard system. Let us write the action of $f$ on the no-slip billiard map as
$f^*C$, so that 
$$ (f^*C)_a:= f_a^{-1}\circ C_{f(a)}\circ f_a. $$
Then $f$ is a symmetry of the no-slip billiard system if $f^*C=C$. We conclude:
\begin{proposition}
Any Euclidean isometry  mapping  $\mathcal{R}$ to itself is a symmetry of the no-slip billiard system.
\end{proposition}

\section{Rolling systems} 
Rolling systems are classical examples of nonholonomic mechanical systems. An especially convenient set of equations describing the motion motion of a rigid sphere with rotationally symmetric mass distribution rolling on submanifolds of Euclidean space (of arbitrary codimension) was obtained in \cite{CFZ21} and is reviewed in this section.
These equations show, in particular, that such systems may be regarded as one-parameter deformations of geodesic flows on hypersurfaces in $\mathbb{R}^{n+1}$, with moment of inertia serving as the deformation parameter.

\subsection{Preliminaries}
By a {\em rolling system} we mean a dynamical system of geometric/mechanical nature consisting of the following elements: a submanifold $P$ of $\mathbb{R}^{n+1}$ possibly with boundary and  corners,
and  arbitrary codimension; an $(n+1)$-dimensional  ball of radius $r$ which rolls over $P$ without slipping\----according to definition to be given shortly\----and may roll over the edge upon reaching the boundary of $P$.
We suppose  the ball has rotationally symmetric mass distribution with the moment of inertia parameter $\gamma$  defined in our discussion of no-slip billiard systems. We further suppose  $r$ to be sufficiently  small, and the boundary of $P$ regular enough, that the boundary $\mathcal{N}=\mathcal{N}_r$ of the 
tubular neighborhood of $P$\----the set of points in $\mathbb{R}^{n+1}$ at a distance less than or equal to $r$ from $P$\----is a piecewise smooth, everywhere differentiable hypersurface. In other words, the unit normal vector field $\nu$ to $\mathcal{N}$, pointing away from $P$, is everywhere continuous and piecewise smooth. Thus it is possible that
the principal curvatures of $\mathcal{N}$ will have jump discontinuities near the boundary of $P$. (See Figure \ref{ball_on_surface}.) Figure \ref{semi_infinite} shows the rolling over the (single point) edge where $P$ is the (codimension-$2$) half-infinite line in $\mathbb{R}^3$.

 The configuration manifold of the rolling ball is the product $SO(n+1)\times\mathcal{N}\subseteq SE(n+1)$. A point$(A,a)$  in this space represents the Euclidean motion mapping the ball
 of radius $r$ centered at $0\in \mathbb{R}^{n+1}$ to the ball centered at $a\in \mathcal{N}$ under the function $x\mapsto Ax+a$.
Due to rotational symmetry, we only need  the position of the ball's center $a\in \mathcal{N}$ at each moment in time in order to specify its configuration. The ball's kinetic state   is given 
by the velocity $v\in T_a\mathcal{N}$ of the ball's center and the angular velocity $U\in \mathfrak{so}(n+1)$.  The condition of no-slip rolling is
\begin{equation}\label{NH_constraint}v=  rU \nu(a). \end{equation}
This is the nonholonomic constraint on the ball's motion asserting that the point on the surface of the ball in contact with $P$ has zero velocity. 

 \begin{figure}[htbp]
\begin{center}
\includegraphics[width=2.5in]{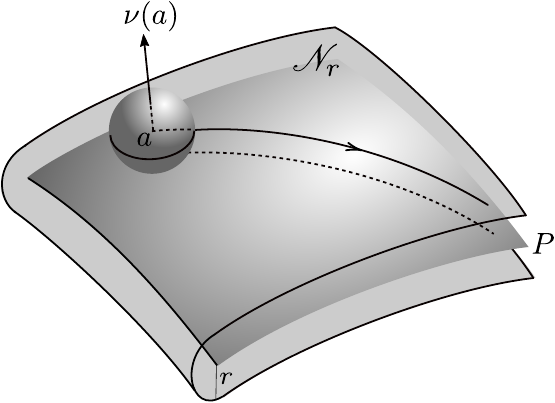}\ \ 
\caption{{\small   $P$ is a smooth submanifold of $\mathbb{R}^{n+1}$ possibly with boundary and corners, and  arbitrary codimension, on which a ball of radius $r$ can roll. The locus of centers of the rolling ball is the  hypersurface $\mathcal{N}=\mathcal{N}_r$, which  is the boundary of an $r$-tubular neighborhood  of $P$. The unit normal vector field to $\mathcal{N}$ is denoted $\nu$.}}
\label{ball_on_surface}
\end{center}
\end{figure} 

As was done for the no-slip billiard ball, we prefer to work not with the actual angular velocity matrix $U$, whose entries have physical dimension $[1/\text{time}]$, but with the matrix $r\gamma U$, whose entries have physical dimension $[\text{velocity}]$. We define
$$S:=r\gamma \Pi_a U \Pi_a \in \mathfrak{so}(T_a\mathcal{N}), $$
where $\Pi_a$ is the orthogonal projection from $\mathbb{R}^{n+1}$ to $T_a\mathcal{N}$ and $\mathfrak{so}(T_a\mathcal{N})$ is
the Lie algebra of skew-symmetric linear maps from $T_a\mathcal{N}$ to itself relative to the standard (dot product) Riemannian metric  on $\mathcal{N}$.
It turns out that the kinetic state of the rolling ball is fully specified
by  $a$ (for the configuration),  $S$ and $v$. In fact, it is not difficult to show (see \cite{CFZ21}) that 
$$U=\Pi_a U\Pi_a +\nu_a\wedge U\nu_a $$
which, together with the constraint equation  (\ref{NH_constraint}), gives
$$r\gamma U = S  +\gamma \nu_a\wedge v. $$
We call $S$ the {\em tangential spin} of the rolling ball. Thus it is natural to look for the equations of motion on a vector bundle $\mathcal{M}$ over $\mathcal{N}$ that extends the tangent bundle $T\mathcal{N}$ by incorporating the tangential spin. We   define 
\begin{equation}\label{direct_sum_M}\pi: \mathcal{M}:= \mathfrak{so}(\mathcal{N}) \oplus T\mathcal{N}\rightarrow \mathcal{N}, \end{equation}
whose fiber over $a\in \mathcal{N}$ is $\mathfrak{so}(T_a\mathcal{N}) \oplus T_a\mathcal{N} $. Sections of $\mathcal{M}$    consist of pairs $(S, X)$ where $X$ is a vector field on $\mathcal{N}$ and $S$ is a field of skew-symmetric endomorphisms on tangent spaces of $\mathcal{N}$.
There is on $\mathcal{M}$  a bundle Riemannian metric
\begin{equation}\label{metricM}\langle \xi,\zeta \rangle_a = m\left\{\frac12 \text{Tr}\left(S_\xi S_\zeta^\intercal\right)+ v_\xi\cdot v_\zeta\right\},\end{equation}
which is naturally associated to the moving ball's kinetic energy: 
$$E_a(\xi)=\frac12\|\xi\|^2:=\frac{1}2 m\left\{\frac12 \text{Tr}\left(SS^\intercal\right)+ |v|^2\right\}, $$
for a state $\xi=(S,v)$ at $a$. Note that the direct sum in (\ref{direct_sum_M}) is orthogonal.

When $n=2$,   $\mathfrak{so}(\mathcal{N})$ is  a trivial line bundle with global section defined at each   $a$ by
the map, denoted $J_a$,  that performs counterclockwise  rotation   by $\pi/2$ on $T_a\mathcal{N}$. In this case, $\mathcal{M}\cong \mathbb{R}\times T\mathcal{N}$ and the tangential spin
reduces to  a scalar quantity $s$.

\subsection{A connection on $\mathcal{M}$}\label{connection}
Elements of the  Euclidean group $SE(n+1)$ have   been written as pairs  $(A,a)$, acting on points $x\in B_r(0)$ (the ball of radius $r$ centered at the origin) by affine transformations $x\mapsto Ax+a$. Thus it was natural to also use   $a$   to designate   points in $\mathcal{N}$.
From now on, as the group and its Lie algebra  take a back seat and we focus attention on 
   $\mathcal{N}$,    points on this manifold will, typically,  be indicated by $x$.

We now introduce a connection  on $\mathcal{M}$ as follows. Recall the  $\wedge$ product (generalizing the cross-product of vector calculus)
 $$\wedge: T_x\mathcal{N}\times T_x\mathcal{N}\rightarrow \mathfrak{so}(T_x\mathcal{N}).$$
 The Riemannian metric induced on $\mathcal{N}$ by restriction of  the dot product in Euclidean space will, when convenient, be written as
 $\langle X,Y\rangle_x$ rather than $X(x)\cdot Y(x)$. 
 \begin{proposition} The Levi-Civita connection on $T\mathcal{N}$ naturally  induces a connection $\nabla$ on $\mathcal{M}$. This connection is metric
 relative to the Riemannian metric (\ref{metricM}) and satisfies the product rule
 \begin{equation}\label{prod_rule_wedge}\nabla_u X\wedge Y=(\nabla_uX)\wedge Y +X\wedge\nabla_uY,\end{equation}
where $X, Y$ are smooth vector fields on $\mathcal{N}$.
 \end{proposition}
 \begin{proof}
The  Levi-Civita connection $\nabla$ naturally extends from vector to general tensor fields on $\mathcal{N}$ and, in particular, to fields of endomorphisms, or $(1,1)$-tensor fields. 
Under this extension, $\nabla$ satisfies the product rule with respect to tensor product and the pairing between vectors and covectors. Thus the product rule (\ref{prod_rule_wedge}) holds since
$$X\wedge Y = Y\otimes X^\flat-X\otimes Y^\flat, $$
where $X^\flat:=\langle X, \cdot\rangle$ is the covector field associated to $X$ via the Riemannian matrix on $\mathcal{N}$. Proving that $\nabla$ is Riemannian on $\mathcal{M}$ only requires checking that
$$X\text{Tr}\left(S_1 S_2\right)= \text{Tr}\left((\nabla_X S_1)S_2\right) +  \text{Tr}\left(S_1(\nabla_X S_2)\right). $$
(The transpose has been dropped since $S^\intercal=-S$.) This can be verified as follows. Let $E_1, \dots, E_n$ be a local orthonormal frame of vector fields on $\mathcal{N}$.  Then
$$\text{Tr}\left(S_1 S_2\right)=\sum_{j} \langle E_j, S_1 S_2 E_j\rangle = -\sum_j\langle S_1E_j, S_2 E_j\rangle.  $$
Now $S_i E_j$ is a local vector field on $\mathcal{N}$ and (from general properties of the covariant derivative on general tensor fields, in particular the 
product rule for contractions)
$$ \nabla_X (S_i E_j)=(\nabla_X S_i)E_j + S_i(\nabla_X E_j).$$
This implies
\begin{align*}-X\langle S_1 E_j, S_2 E_j\rangle&= \left\langle S_2 (\nabla_X S_1)E_j, E_j\right\rangle +  \left\langle S_1 (\nabla_X S_2)E_j, E_j\right\rangle \\ 
&\ \ \ \ \ \ \ \ \ \ \ \ \ \ \ \ \ \ \ \ + \langle S_1\nabla_X E_j, S_2 E_j\rangle + \langle S_1 E_j, S_2 \nabla_X E_j\rangle.
\end{align*}
Summing over $j$,
\begin{align*}X\text{Tr}(S_1S_2)&= \text{Tr}((\nabla_XS_1)S_2) +  \text{Tr}(S_1(\nabla_XS_2))\\
& \ \ \ \ \ \ \ \ \ \ \ \ \  + \sum_j \left( \langle S_1\nabla_X E_j, S_2 E_j\rangle + \langle S_1 E_j, S_2 \nabla_X E_j\rangle\right). 
\end{align*}
It remains to show that the sum in the previous identity vanishes. Substituting
$$\nabla_X E_j= \sum_k \langle E_k, \nabla_X E_j\rangle E_k$$ into that sum (indicated by $\mathcal{I}$) results in
$$\mathcal{I}=\sum_{j,k} \langle E_k, \nabla_X E_j\rangle \left(\langle S_1 E_k, S_2E_j\rangle +\langle S_1 E_j, S_2E_k\rangle\right).$$
But it is now apparent that $\mathcal{I}$ is the trace of the product of a skew-symmetric matrix and a symmetric matrix, which is $0$. 
 \end{proof}
 
 \subsection{The free rolling equations}\label{equations_M_section}
The {\em shape operator} of the hypersurface $\mathcal{N}$ is defined at each $x\in \mathcal{N}$ by the symmetric linear map 
$$\mathbb{S}_x: u\in T_x\mathcal{N}\mapsto \mathbb{S}_xu:=-D_u\nu\in T_x\mathcal{N},$$
in which $D_u \nu$ is the directional derivative of the unit normal vector field $\nu$ in the direction of $u\in T_x\mathcal{N}$. 
The eigenvectors of $\mathbb{S}_x$ are the {\em principal directions} at $x$, and the associated eigenvalues are the {\em principal curvatures} of $\mathcal{N}$ at $x$. 

We can now describe the equations of motion for a ball with rotationally symmetric mass distribution and moment of inertia parameter $\gamma$ in a form that
does not require explicit use of the nonholonomic constraint equation.  It will be convenient to introduce yet another proxy for the moment of inertia: 
\begin{equation}\label{eta_gamma}\eta:=\frac{\gamma}{\sqrt{1+\gamma^2}}.\end{equation}
The following theorem is from \cite{CFZ21}.
\begin{theorem}\label{rolling_theorem}
The motion of a rolling ball  of radius $r$ on the  submanifold $P\subseteq \mathbb{R}^{n+1}$ with moment of inertia parameter $\eta$ is given by a curve
$x(t)\in \mathcal{N}_r$ and a differentiable section $(S(t), v(t))$ of $\mathcal{M}$ such that $v(t)=x'(t)$ and
$$\frac{\nabla v}{dt}=-\eta S\mathbb{S}_{x} v, \ \ \ \frac{\nabla S}{dt} = \eta (\mathbb{S}_x v)\wedge v.$$
\end{theorem}

Let us make a few observations about this set of equations.
\begin{enumerate}
\item When the moment of inertia parameter $\eta$ is zero, the two equations decouple and the system reduces to
$$\frac{\nabla x'}{dt} = 0, \ \ \ \frac{\nabla S}{dt}=0. $$
This means that the trajectories of the center of the rolling ball are geodesics in $\mathcal{N}$ and the tangential spin $S(t)$ is a parallel tensor field along those geodesics. This is reminiscent of frame flows on Riemannian manifolds, except that it is not a frame but some kind of frame spinning  that undergoes parallel transport.
Thus we can regard the above rolling equations as a one-parameter deformation of (a bundle extension of) the geodesic flow on the tangent bundle of $\mathcal{N}$.
See Figure \ref{semi_infinite}.

 \begin{figure}[htbp]
\begin{center}
\includegraphics[width=5.0in]{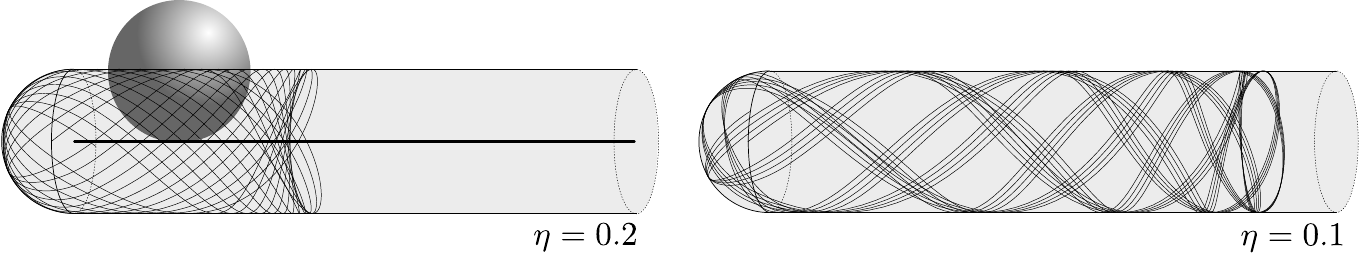}\ \ 
\caption{{\small A $3$-dimensional ball  rolling without slipping  on a half-infinite line $P$ in $\mathbb{R}^3$, showing the rolling at the boundary of $P$. As the moment of inertia parameter $\eta$ is reduced, trajectories approach geodesics on the boundary $\mathcal{N}_r$ of the $r$-tubular neighborhood of $P$. }}
\label{semi_infinite}
\end{center}
\end{figure} 
\item  The dynamical properties of the system depend strongly on the extrinsic curvatures of $\mathcal{N}$ via the shape operator. When $\mathbb{S}=0$, 
the center of the rolling ball follows uniform rectilinear motion (on the  flat $\mathcal{N}$) while $S(t)$ is parallel;  hence the components of $S(t)$ are  conserved quantities.
\item Let us define the  bundle map $f:\mathcal{M}_x\rightarrow \mathcal{M}_x$ at each $x\in \mathcal{N}$ (in which $\mathcal{M}_x$ indicates the fiber of $\mathcal{M}$ over $x$) by
$$f(S,v)= \left((\mathbb{S}_x v)\wedge v, -S\mathbb{S}_x v \right). $$
Then the rolling equations may be written as
$$\frac{\nabla \xi}{dt} = \eta f(\xi).  $$
It is an easy exercise to check that $\langle \xi, f(\xi)\rangle=0$, from which it follows that
$$\frac{d}{dt} E(\xi(t)) = \frac12 \frac{d}{dt}\langle \xi(t),\xi(t)\rangle =  \left\langle \xi(t), \frac{\nabla\xi}{dt}(t)\right\rangle =\langle \xi(t), f(\xi(t))\rangle=0.$$
This shows  conservation of energy for the rolling equation.
\item It is not difficult to verify that the rolling equations are time-reversible.
\item The vector bundle $\mathcal{M}$ can be identified with the tangent bundle of a Riemannian manifold $M$, as follows.   
Let  $\pi:M\rightarrow \mathcal{N}$ be the fiber bundle of linear isometries on tangent spaces of $\mathcal{N}$. By definition, the fiber $M_x$ over $x\in \mathcal{N}$
consists of  all the orientation preserving linear maps $A:T_x\mathcal{N}\rightarrow T_x\mathcal{N}$ such that $\langle Au, Av\rangle_x=\langle u,v\rangle_x$
for all $x\in \mathcal{N}$ and all $u, v\in T_x\mathcal{N}$. In other words, $M$ is a bundle of groups, the fiber over $x$ being the special orthogonal group
on the inner product vector space $(T_x\mathcal{N}, \langle \cdot, \cdot\rangle_x).$ We denote this group by $SO(\mathcal{N})_x$ and its Lie algebra
by $\mathfrak{so}(\mathcal{N})_x$. The tangent bundle of $M$ can be identified with $\mathcal{M}$ as follows. Let $e\in TM$ have base-point $A\in M$,
where $A\in SO(\mathcal{N})_x$. Let $t\mapsto A(t)\in M$ be a differentiable curve representing $e$ in the sense  that $A(0)=A$ and $A'(0)=e$.
We can then map $e$ to the pair $(S,u)$ where
$$u=\left.\frac{d}{dt}\right|_{t=0}\pi(A(t))\in T_x\mathcal{N}, \ \ S=r\gamma \left(\left.\frac{\nabla}{dt}\right|_{t=0} A(t)\right) A^{-1}\in \mathfrak{so}(\mathcal{N})_x. $$
Here $\nabla$ is the covariant derivative  obtained through the imbedding $M\subseteq T\mathcal{N}\otimes T^*\mathcal{N}$. Through the identification $TM\cong \mathcal{M}$, the manifold $M$ acquires a Riemannian metric and a metric connection $\nabla$. It is natural to ask whether  $\nabla$ is the Levi-Civita connection
on $M$, which  amounts to determining whether the torsion tensor of $\nabla$ vanishes. This is a question we have not yet addressed. 
\end{enumerate}
\subsection{Rolling under external forces}
We briefly remark on how external forces can be incorporated into the rolling equations of Theorem \ref{rolling_theorem}. Let $\varphi{x}$ be a force (per unit mass) vector field
on $\mathbb{R}^{n+1}$. This means that $\varphi(Ax+a)d\mu(x)$ is the force vector acting on an element of mass $d\mu(x)$ of the ball in configuration $g=(A,a)$, where
$x$ is here a point on the ball $B_r(0)$ of radius $r$ centered at the origin. We obtain a linear functional on the tangent space at $g$ to  the configuration manifold
$SO(n+1)\times \mathcal{N}$ of the rolling ball:
$$\hat{\varphi}_g(\xi):=\int_{B_r(0)} V_\xi(x)\cdot\varphi(Ax+a)\, d\mu(x).$$
Recall that $V_\xi(x)\in \mathbb{R}^{n+1}$ denotes the velocity of the material point $x\in B_r(0)$ in configuration $g$ and kinetic state $\xi\in T_g(SO(n+1)\times \mathcal{N})$. We call $\hat{\varphi}$ the {\em power functional}. Under the assumption that $\mu$ is rotationally symmetric, we obtain (using the trace inner product and
notations from Section \ref{generalities})
$$ \hat{\varphi}_g(\xi)=\left\langle \left(U_\xi, u_\xi\right), \left(\Psi(a), \psi(a)\right) \right\rangle_g,$$
where  
$$\Psi(a):=\frac1{(r\gamma)^2} \int_{B_{r}(0)} x\wedge \varphi(x+a)\, \frac{d\mu(x)}{m},
\ \  \psi(a):=
 \int_{B_{r}(0)}   \varphi(x+a)\, \frac{d\mu(x)}{m}.
$$
Going through the argument from \cite{CFZ21} used to derive the free equations of motion   of Theorem \ref{rolling_theorem}, but now taking account of the external forces, which affect the
determination of the rolling constraint force, one obtains (after switching once again the roles of $a$ and $x$; see the beginning of Section \ref{connection}) the equations
\begin{equation}\label{motion_force}\frac{\nabla v}{dt}=-\eta S\mathbb{S}_{x} v+   r\eta^2 \Psi \nu + \frac1{1+\gamma^2}\Pi_x\psi, \ \ \ \frac{\nabla S}{dt} = \eta (\mathbb{S}_x v)\wedge v +{r\eta}\Pi_x\Psi\Pi_x.\end{equation}
Here $\Pi=\Pi_x$ is the orthogonal projection map to $T_x\mathcal{N}$. Of interest in this paper is the case of a constant $\varphi$, for which 
$\Psi=0$ and $\psi=\varphi$. 

\section{No-slip billiards as limits of rolling systems}
We can finally indicate the connection between no-slip billiards and rolling systems. The key remark is based on determining the outcome  of rolling around a straight edge. 
Thus suppose that $P$ is the  (codimension-$1$) half-space $\mathbb{R}^{n-1}\times [0,\infty)$ in $\mathbb{R}^{n+1}$, the boundary of which is  $\mathbb{R}^{n-1}$.
Rolling around the edge then corresponds to the rolling of an $(n+1)$-dimensional ball on $\mathbb{R}^{n-1}$. The curved part of the hypersurface $\mathcal{N}_r$ is
part of the boundary of the tubular neighborhood of $\partial P$, which is
is $\mathbb{R}^{n-1}\times S^1(r)$, where $S^1(r)$ is the circle of radius $r$ on the plane perpendicular to $\mathbb{R}^{n-1}$ in $\mathbb{R}^{n+1}$. See
Figure \ref{edge}.

In order to solve the rolling equations in this setting (Theorem \ref{rolling_theorem}), it is convenient to express $v$ and $S$ in the parallel orthonormal frame described in
Figure \ref{edge}. Thus we write
$$v=\sum_i v_i E_i, \ \ S=\frac12\sum_{i,j} S_{ji} E_i\wedge E_j, \ \ S_{ji}:=E_i\cdot (S E_j). $$
Since the $\nabla_vE_i=$ are parallel and $\mathbb{S}_xE_i=-\frac1r\delta_{ni}E_i$ (on the curved part of $\mathcal{N}$), we have
$$\frac{\nabla v}{dt} =\sum_i\dot{v}_i E_i = -\eta S\mathbb{S}_xv =\frac{\eta v_n}{r} SE_n,$$  
where $\dot{v}_i$ is derivative in the time variable. Taking the inner product with $E_i$ gives
$$\dot{v}_n=0, \ \ \dot{v}_i= \frac{\eta v_n}{r} S_{ni} \text{ for } i<n.  $$
Note that $v_n$ is constant because $S_{nn}=0$.  Similarly, (as the $E_i$ are parallel)
$$\dot{S}_{ij}=E_i\cdot \left(\frac{\nabla S}{dt} E_j\right) = \eta E_i\cdot\left((\mathbb{S}_x v) \wedge v E_j\right)=-\frac{\eta}{r} E_i\cdot(E_n\wedge v E_j) =-\frac{\eta}{r}(\delta_{nj} v_i - \delta_{ni}v_j).$$
Thus we are left with the elementary problem of  solving  the system
$$\dot{v}_n=0, \ \  \dot{S}_{ij}=0 \text{ for } i,j<n, \ \ \dot{S}_{ni} =-\frac{\eta v_n}{r} v_i   \text{ and }  \dot{v}_i =\frac{\eta v_n}{r} S_{ni}  \text{ for } i<n.$$
Let us solve for the velocities immediately after leaving the round edge as a function of the velocities immediately before entering it. 

 \begin{figure}[htbp]
\begin{center}
\includegraphics[width=2.2in]{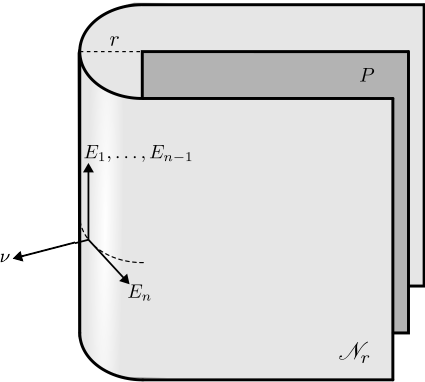}\ \ 
\caption{{\small  The unit vector fields $E_1, \dots, E_n$ form a parallel orthonormal frame on $\mathcal{N}$ consisting of principal directions. $E_n$
is associated to principal curvature $-1/r$ whereas the principal curvatures for the other vectors are $0$. }}
\label{edge}
\end{center}
\end{figure} 

We define $$\overline{v}:=\sum_{i=1}^{n-1} (E_i\cdot v) E_i, \ \ 
  W:=SE_n.$$ Note that both $\overline{v}$ and $W$ lie in    $\mathbb{R}^{n-1}$.
These are the velocity components that do not remain constant in the process of rounding the edge of the flat plate $P$. The time it takes to go around the edge
is $T=\pi r/v_n$. Integrating from $0$ to $T$ the system
$$\frac{d}{dt}\left(\begin{array}{c}\overline{v} \\W\end{array}\right) =\frac{\eta \pi}{T}\left(\begin{array}{cc}0 & I \\\!\!-I & 0\end{array}\right) \left(\begin{array}{c}\overline{v} \\W\end{array}\right)$$
gives
\begin{equation}\label{map_rolling_edge}\left(\begin{array}{c}\overline{v}(T) \\W(T)\end{array}\right) =\left(\begin{array}{rr}\cos(\pi \eta)I & -\sin(\pi \eta)I \\\sin(\pi\eta)I & \cos(\pi\eta)I\end{array}\right) \left(\begin{array}{c}\overline{v}(0) \\W(0)\end{array}\right). \end{equation}
Let us now change our perspective somewhat and imagine looking at the folded hypersurface of Figure \ref{edge} from above and seeing its two sheets connected by the half-cylinder as one. From this perspective, the rolling around the edge is perceived as a soft collision with a flat wall. We rewrite the transformation of Equation (\ref{map_rolling_edge}) accordingly, as follows.
The vector field $E_n$ is viewed as pointing out on the edge ($\cong \mathbb{R}^{n-1}$) of the lower flat part of $\mathcal{N}$, and pointing in at the edge of the upper flat part. If we now identify both edges and replace $E_n$ there  with the unit outward pointing normal vector field $\mathbbm{n}$, redefine $W$ as $S\mathbbm{n}$,  and
introduce $\hat{v}:=v\cdot\mathbbm{n}$, the transformation (\ref{map_rolling_edge}) becomes
$$\left(\begin{array}{c}\hat{v} \\\overline{v} \\W\end{array}\right)\mapsto \left(\begin{array}{ccc}-1 & 0 & 0 \\0 & \cos(\pi\eta)I &  \sin(\pi\eta)I  \\0 &  \sin(\pi\eta)I   &\!\!\!\! - \cos(\pi\eta)I \end{array}\right)\left(\begin{array}{c}\hat{v} \\\overline{v} \\W\end{array}\right). $$
This expression should be compared with the corresponding map for the no-slip billiard collision, Equation (\ref{collision_map_bis}).
They become identical if we equate the two moment of inertia parameters
$ \beta = \pi \eta.$ Returning to the more basic parameter $\gamma$, and writing $\gamma_{n}$ and $\gamma_{n+1}$ for those parameters in dimensions $n$ and $n+1$, we have
\begin{equation}\label{compare_gamma}\frac{1-\gamma_n^2}{1+\gamma_n^2} =\cos\left( \frac{\gamma_{n+1}}{\sqrt{1+\gamma_{n+1}^2}}\pi\right).\end{equation}
 When $r$ is very small, so that the time the ball spends in contact with the edge of $P$ can be neglected, and if the moment of inertia parameters are
  related as indicated above, then the
 rolling-around-the-edge map in dimension $n+1$ becomes indistinguishable from the
  no-slip collision map for the no-slip billiard system in dimension $n$. 
 If the edge of $P$ is curved, the determination of the outgoing velocities away from the edge of $P$ as a function of the incoming velocities will be much more difficult, as will be seen in later examples. But it can be expected that, as the radius of the ball approaches $0$, the principal curvatures of the edge of $P$ as perceived by the rolling ball become negligible, and the flat approximation can be used.  This is the key idea behind the proof of the following theorem from \cite{CFZ21}.

\begin{theorem}[\cite{CFZ21}]\label{main}
Let $P$ be a domain in $\mathbb{R}^n\cong \mathbb{R}^n\times\{0\}\subseteq\mathbb{R}^{n+1}$ with smooth boundary $P_0$ whose principal curvatures (as a hypersurface in $P$) are uniformly bounded.
Consider an $(n+1)$-dimensional ball with spherically  symmetric mass distribution constrained to roll without slip on $P$. We assume no forces other than 
those enforcing the constraint. 
  In the limit
when the radius of the ball approaches $0$, solutions of the rolling ball equation have the following description: On $P\setminus P_0$ the point-mass moves with constant velocities $v$ and $S$ (this is also true for positive radius);  upon reaching a boundary point $x\in P_0$ with unit normal vector $\mathbbm{n}(x)$, the vector $(v, W)\in T_xP\oplus T_xP_0$, with $W=S \mathbbm{n}(x)$,   undergoes a reflection according to the no-slip  collision map $C_x$ as described in Equation (\ref{collision_S}).
The moment of inertia parameters $\gamma_{\text{\tiny no-slip}}$ (going into $C_x$) and $\gamma_{\text{\tiny roll}}$ (defining $\eta$ in the rolling equations), bothindependent of the ball radius, are related to each other by
Equation (\ref{compare_gamma}).
\end{theorem}

In the presence of   external forces, the segments of  trajectories on the flat part of $P$ will, naturally, no longer follow a uniform rectilinear motion, but such forces
do not affect what happens at the boundary in the limit $r\rightarrow 0$. 

\begin{definition}[Nonholonomic billiard system]
By a {\em nonholonomic billiard system} we mean a rolling system on a flat plate $P$ as in Theorem \ref{main}, for a ball with positive radius, rotationally symmetric mass distribution, with or without external forces. 
\end{definition}

We may think of nonholonomic billiard systems as soft versions of no-slip billiards.

 \subsection{Rolling on vertical strips under gravity}\label{vertial_strip_g}
 We are now in a position to begin addressing  the main point of this paper, which is to use the nonholonomic billiards, governed by systems of ordinary differential equations, as a tool to  help illuminate and explain some of the observed behavior of no-slip billiards, in particular the bounded trajectories on cylinders under gravity. 
  This is a project we only begin here; it will be  systematically developed elsewhere.
 In this section, we consider the case of a no-slip disc bouncing between two parallel vertical lines (Figure \ref{plates}).
 
  \begin{figure}[htbp]
\begin{center}
\includegraphics[width=1.5in]{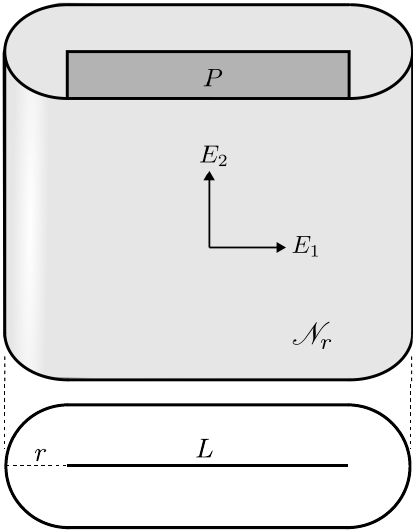}\ \ 
\caption{{\small  Segment of the cylindrical  surface $\mathcal{N}_r$ associated to an infinite vertical flat strip $P$ in $\mathbb{R}^3$. Constant acceleration due to gravity is $-gE_2$, where $E_1, E_2$ form a parallel orthonormal frame on $\mathcal{N}_r$.}}
\label{NH_strip}
\end{center}
\end{figure} 
 
The following fact, from \cite{CCCF20}, lends support to the above expressed hope for the usefulness of the nonholonomic billiards. It deals with rolling in  vertical cylinders in dimension $3$ under constant gravity.

\begin{theorem}[\cite{CCCF20}]\label{theorem_bounded_3D}
Suppose the cross-section of the $3$-dimensional vertical cylinder is a differentiable simple closed curve and that the initial velocity of the center of the rolling ball
has nonzero horizontal (i.e., cross-sectional) component. Then the trajectories of the rolling motion under a constant force parallel to the axis of the cylinder are bounded. 
\end{theorem}

 Let us explore this fact for $P\subseteq \mathbb{R}^3$ consisting of the vertical strip $P$ shown in Figure \ref{NH_strip}. Let $E_1, E_2$ be the parallel orthonormal
 vector fields indicated in the figure and write
 $v=v_1E_1+v_2 E_2$, $S=s J$, $J=E_1\wedge E_2$. 
 Note that $E_1$ and $E_2$ are principal directions with respective principal curvatures   $\kappa(x)$ and $0$, where $\kappa=0$ on the front and back flat strips of $\mathcal{N}_r$ and $\kappa=-1/r$ on the two half-cylinders. 
  The acceleration due to gravity is the constant $g$. Equations (\ref{motion_force}) reduce to the system
\begin{align*}
\dot{v}_1 & = 0\\
\dot{v}_2 & = -\eta v_1 \kappa(x(t)) s -g\\
\dot{s} &= \eta v_1 \kappa(x(t))v_2.
\end{align*}
Let us set $v_1=\text{constant}=u$. The motion of the center of the ball (in $\mathcal{N}_r$) projects to  the cross-section in $\mathbb{R}^2$ to uniform motion, with
speed $u$, along the stadium shaped curve $t\mapsto (x_1(t), x_2(t))$.   Due to translation symmetry, the principal curvature $\kappa$ only depends on $(x_1, x_2)$.
We write $k(t):=\kappa(x_1(t),x_2(t))$.

\begin{figure}[htbp]
\begin{center}
\includegraphics[width=3.0in]{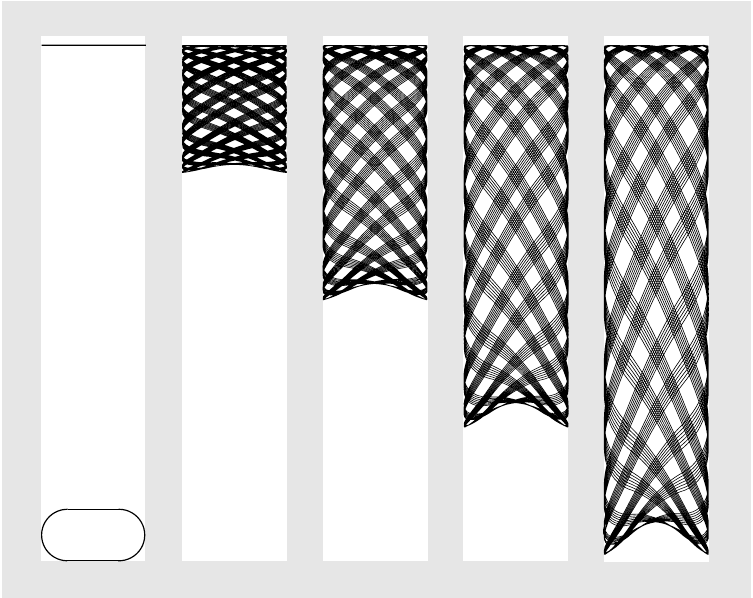}\ \ 
\caption{{\small Rolling on a vertical strip under gravity; $\mathcal{N}_r$ is a cylinder with a stadium shaped cross-section (see bottom-left here, or Figure \ref{NH_strip}), with $L=1$ and $r=0.5$.  Acceleration $g$ increases from left to right ($g=0$ on the far left). These trajectories should be compared with those of the no-slip billiard system of Figure \ref{plates}, illustrating the idea of
nonholonomic billiards as  smooth approximations of the no-slip kind. (No attempt was made to match parameters.)}}
\label{NH_vert_strip}
\end{center}
\end{figure}

Thus we are left to solve the time-dependent system
$$
\dot{v}_2= -\eta u k(t) s -g, \ \ 
\dot{s}    = \eta u k(t) v_2.
$$
This may be written in matrix form as
$$ \frac{d}{dt}\left(\begin{array}{c}v_2 \\s\end{array}\right) =\eta u k(t) \left(\begin{array}{cc}0 &\!\! -1 \\1 & 0\end{array}\right) \left(\begin{array}{c}v_2 \\s\end{array}\right)-g\left(\begin{array}{c}1 \\0\end{array}\right)$$
This is  easily solved by elementary methods. Let  $J$ be the $2$-by-$2$ matrix on the right-hand side of this equation and define
$$N(t):=\eta u \int_0^tk(s)\, ds, \ \ I(t):=\exp\left\{ -N(t)J \right\}. $$
Then the solution vector $X(t)=(v_2(t),s(t))^\intercal$ satisfies
$$X(t) =I(t)^{-1}X(0) - I(t)^{-1}\int_0^t I(s)G\, ds,$$
where $G=g(1,0)^\intercal$.
Solution curves for the center of the rolling ball are shown in Figure \ref{NH_vert_strip}. Here we see that boundedness of trajectories follows from the  general
result of Theorem \ref{theorem_bounded_3D}. Passing to the limit as $r$ approaches zero, we recover boundedness of trajectories of the no-slip billiard system on the vertical strip. (Figure \ref{NH_strip}.)

This example serves as a simple model for higher dimensional   cylinders and  provides some justification for  our  proposal of using rolling systems in one extra dimension to illuminate the dynamics of no-slip billiards. The $3$-dimensional case,  requiring rolling in dimension $4$, is naturally more challenging. The remainder of this paper indicates  the first steps in this direction. 
 
\section{Rolling over $3$-dimensional cylinders  in $\mathbb{R}^4$} \label{section_cylinder}
 In this section  we prepare the groundwork for a study of a $4$-dimensional ball rolling on a $3$-dimensional solid cylinder $P$ in $\mathbb{R}^3$. The immediate goal is
 to obtain the system of differential equations governing the motion. Those equations will have a hierarchical structure. The rolling hypersurface $\mathcal{N}_r$
 is a product $\overline{\mathcal{N}}_r\times \mathbb{R}$ where the cross-section  $\overline{\mathcal{N}}_r$ is  what  we call a {\em pancake surface} in $\mathbb{R}^3$. 
 The projection  of the rolling motion on this surface  satisfies  equations that  do not depend on the axial variable (height);  the cross-sectional trajectories then feed into a set of differential equations for the height function. It was seen above that, for the  rolling on a vertical strip in dimension $3$,   the cross-sectional motion was  very simple\----uniform motion on a closed simple curve\----while the motion on pancake surfaces can be much richer dynamically, as we will see shortly.
 
 \subsection{Some geometric definitions}
 A few  definitions are needed first.  Whenever convenient, we regard $\mathbf{x}=(x_1, x_2)\in \mathbb{R}^2$ as a point in $\mathbb{R}^3$ or $\mathbb{R}^4$ by the identification
$$ \mathbb{R}^2= \{(x_1, x_2, 0, 0)\}\subseteq \mathbb{R}^3=\{(\mathbf{x},x_3,0)\}\subseteq \mathbb{R}^4.$$
The cylinder $P$ is defined by its cross-section $\mathcal{C}$, a region in $\mathbb{R}^2$ whose boundary is differentiable, piecewise smooth, and has bounded curvature. 
(The assumption that this boundary is differentiable can
be relaxed so as to include polygonal regions; in this case, we require the wedge at each vertex made of tangent half-lines to be convex.)
We do not require
the boundary  of $\mathcal{C}$ to be connected; for example, it could be a strip bound by two parallel straight lines. Now set $P=\mathcal{C}\times\mathbb{R}$.
We may parametrized the boundary of $\mathcal{C}$ locally by arclength. If $\gamma(s)$ is such parametrization, then
$$e_2(\gamma(s)) := \gamma'(s), \ \ e_1(\gamma(s)) := -J e_2(\gamma(s)).$$
Here, $J$ is, at each point on the plane, the linear operation that rotates tangent vectors at that point by $\pi/2$ in the positive direction. See Figure \ref{gen_cyl} for
the orientation used. Further, 
let $e_3=(0,0,1,0)$, $e_4=(0,0,0,1)$.    

  \begin{figure}[htbp]
\begin{center}
\includegraphics[width=4.0in]{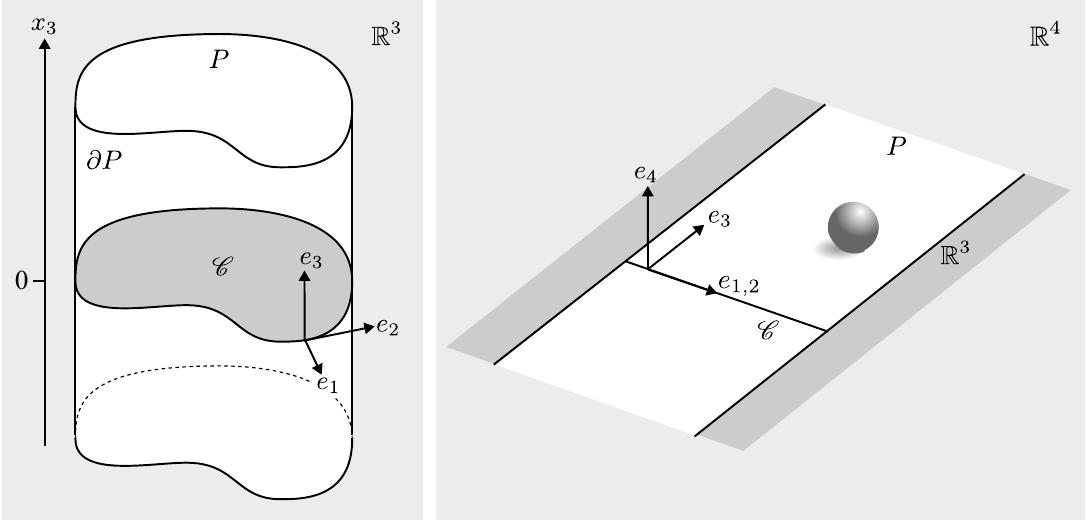}\ \ 
\caption{{\small On the left,  a general cylinder $P=\mathcal{C}\times \mathbb{R}$ in $\mathbb{R}^3$ with cross-section $\mathcal{C}$. On the right: schematic picture of the cylinder in $\mathbb{R}^4$ on which the $4$-dimensional ball rolls. The figure shows, in particular, the frame field $e_1, e_2, e_3, e_4$ adapted to the
cross-section contour.}}
\label{gen_cyl}
\end{center}
\end{figure}

Let $$P_r=\{ x\in \mathbb{R}^4: \text{dist}(x,P)\leq r\},$$ for $r>0$.  For sufficiently small $r$, the map $\pi_P:P_r\rightarrow P$  sending $x\in P_r$ to its closest point in $P$ is well-defined and differentiable, and the boundary of $P_r$, denoted
$$\mathcal{N}:=\mathcal{N}_r:=\{x\in \mathbb{R}^4: |x-\pi_P(x)|=r\}, $$
is a differentiable $3$-dimensional manifold without boundary.
Given $x=(\mathbf{x},x_3,x_4)\in P_r$, 
 $$\pi_P(x)=(\mathbf{y},x_3, 0)$$ and 
$\mathbf{y}=\mathbf{x} \text{ if } \mathbf{x}\in \mathcal{C}$, and $\mathbf{y}\in \partial \mathcal{C} \text{ if } \mathbf{x}\notin \mathcal{C}.$
It is also easily seen that  $\pi_P(x+se_3)=\pi_P(x)+se_3$.
Note that the vector fields  $e_1$ and $e_2$, initially defined on the boundary of $\mathcal{C}$,  extend to the complement in  $P_r$ of the interior of $P$:
$$e_1(x):=\frac{\mathbf{x}-\mathbf{y}}{|\mathbf{x}-\mathbf{y}|}, \ \ e_2(x)=Je_1(x). $$
Thus we have the adapted frame $e_1, e_2, e_3, e_4$ defined on that extended set. 

The hypersurface $\mathcal{N}$  decomposes as $\mathcal{N}=\mathcal{N}^+\cup \mathcal{N}^-\cup \mathcal{N}^c$ where 
$$\mathcal{N}^\pm:=\{x\in \mathbb{R}^4:\mathbf{x}\in \mathcal{C}\text{ and } x_4=\pm r\}, \ \ \mathcal{N}^c:=\{x\in \mathbb{R}^4: \mathbf{x}\notin \mathcal{C} \text{ and } |x-\pi_P(x)|=r\}. $$ 
We refer to $\mathcal{N}^\pm$ as the {\em flat parts} of $\mathcal{N}$ and to $\mathcal{N}^c$ as the {\em curved part}. See Figure \ref{N_surface}.

The unit normal vector field to $\mathcal{N}$ is defined at a point $x$ on the hypersurface by
$$ \nu(x):=\begin{cases}
\pm e_4 & \text{ if } x\in \mathcal{N}^\pm\\
 (x-\pi_P(x))/r & \text{ if } x\in \mathcal{N}^c.
\end{cases}$$
For $x$ on the curved part $\mathcal{N}^c$, define the unit tangent vector 
 $\tau(x)$ perpendicular to $e_2$ and $e_3$, with the additional property that if $x\in \mathcal{N}^\pm\cap \mathcal{N}^c$, then its projection
 to $\mathbb{R}^3$ equals $\pm e_1(x)$. 
 Note that the  $\nu(x)$ and  $\tau(x)$  and  the pair $e_1(x)$ and  (the constant) $e_4(x)$ both linear span the same plane and  
 the vector fields $\tau$, $e_2$ and $e_3$ constitute an orthonormal frame on the tangent bundle of $\mathcal{N}^c$.
When convenient, we use the more uniform notation $X_1=\tau$, $X_2=e_2$, $X_3=e_3$.

We call an $x_3$-slice of $\mathcal{N}_r$ a {\em pancake surface}. The $0$-slice is $\overline{\mathcal{N}}_r$ and $\mathcal{N}_r=\overline{\mathcal{N}}_r\times \mathbb{R}$. Figure \ref{Sinai_pancake} shows, on the right, the pancake surface associated to a cylinder $P$ whose cross-section is a Sinai billiard table. This is 
a torus with an open disc removed.

 \subsection{A parametrization of $\mathcal{N}^c$}
 Let $\varphi$ be the angle shown in Figure \ref{N_surface}, $0\leq \varphi\leq \pi$. 
Then
$$\nu(x)=\sin(\varphi)e_1(x)+\cos(\varphi)e_4, \ \  \tau(x)=\cos(\varphi) e_1-\sin(\varphi) e_4.$$
Given a local parametrization $s\mapsto \gamma(s)$   of the boundary of $\mathcal{C}$ and $\varphi$, we
obtain a local parametrization of $\mathcal{N}^c$ as
$$ x=a(s,\varphi, x_3)= \gamma(s)+ r\left[\sin(\varphi) e_1(\gamma(s)) + \cos(\varphi) e_4\right] +x_3 e_3.$$
Let $\kappa(s)$ be the curvature of $s\mapsto \gamma(s)$, defined by the expression
$$ \frac{d}{ds}e_1(\gamma(s))=-\kappa(s) e_2(\gamma(s)).$$

 \begin{figure}[htbp]
\begin{center}
\includegraphics[width=3.0in]{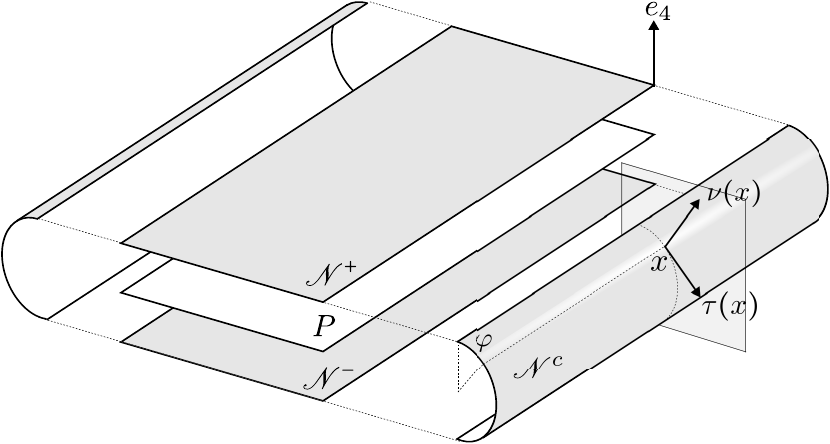}\ \ 
\caption{{\small Decomposition of the hypersurface $\mathcal{N}$ into its flat and curved parts.}}
\label{N_surface}
\end{center}
\end{figure} 

For example, if $\mathcal{C}$ is a disc of radius $R$, then $\kappa=-1/R$.
Relative to this parametrization, the vector fields $X_1=\tau$, $X_2=e_2$ and $X_3=e_3$ (expressed as first order differential operators on functions on $\mathcal{N}$) assume the form
$$X_1(s,\varphi,x_3)= \frac1r \frac{\partial}{\partial \varphi}, \ \ X_2(s,\varphi,x_3)=\frac{1}{1-r\kappa(s)\sin\varphi} \frac{\partial}{\partial s}, \ \ X_3(s,\varphi,x_3)=\frac{\partial}{\partial x_3}$$
as is easily shown.  Note that $\kappa(s)\leq 0$ if $\mathcal{C}$ is convex. The need to make $r$ sufficiently small in order to avoid focal points in $\mathcal{N}_r$
is seen clearly in the expression of $X_2$. The following proposition, in which $\nabla$ indicates the Levi-Civita connection on $\mathcal{N}$ for the Riemannian metric induced from the dot product in $\mathbb{R}^4$, summarizes some of the basic properties of these vector fields. Recall that the {\em shape operator} is defined by
 $$\mathbb{S}_x v:=-D_{v}\nu$$ for  $v\in T_x\mathcal{N}^c$, where $D$ indicates the Euclidean directional derivative of vector fields.

\begin{figure}[htbp]
\begin{center}
\includegraphics[width=3.0in]{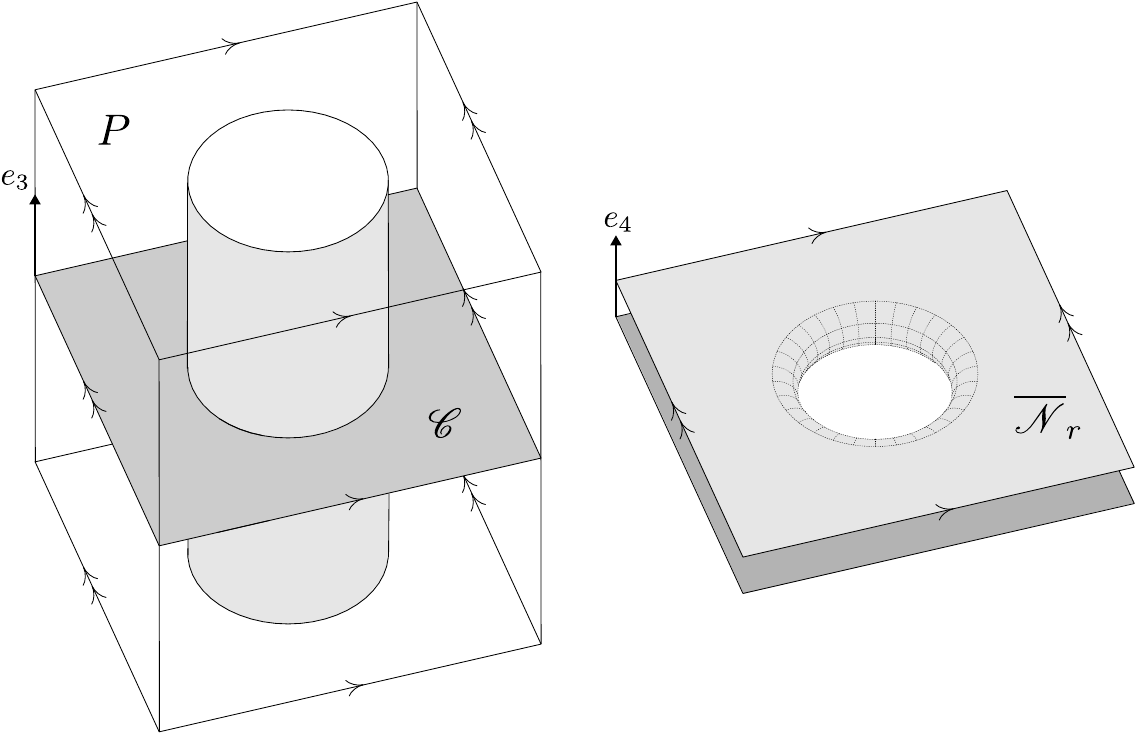}\ \ 
\caption{{\small  Left:  $P$ is a cylinder whose cross-section is a Sinai billiard table; right:   pancake $x_3$-slice of $\mathcal{N}_r$. Vertical direction: $e_3$ on the left, $e_4$ on the right. Compare with Figure \ref{Sinai_compare}.}}
\label{Sinai_pancake}
\end{center}
\end{figure}

\begin{proposition} \label{calculations_frame} Define the functions
$$f_c(s,\varphi)=\frac{\kappa(s)\cos\varphi}{1-r\kappa(s) \sin \varphi}, \ \ f_s(s,\varphi)=\frac{\kappa(s)\sin\varphi}{1-r\kappa(s) \sin \varphi}. $$
The vector fields $X_1, X_2, X_3$ satisfy the following properties. Their Lie brackets are
$$ [X_1, X_2]=f_c(s,\varphi)X_2, \ \ [X_1, X_3]=[X_2,X_3]=0.$$
Their covariant derivatives are
$$\nabla_{X_2}X_2 =f_cX_1,\ \ \nabla_{X_2}X_1 = -f_c(s,\varphi) X_2, \ \ \nabla_{X_1}X_j=0, \ \ \nabla_{X_3}X_j=0, \ \ \nabla_{X_j}X_3=0 $$
 for all  $j$. At each point,  $X_1, X_2, X_3$ are principal directions for  the shape operator:
$$\mathbb{S}_x X_1 =-\frac1r X_1, \ \  \mathbb{S}_x X_2 =f_s(s,\varphi)X_2, \ \ \mathbb{S}_x X_3=0.$$
Finally, 
$ \nabla_v (X_1\wedge X_2)=0$ and 
$$
\nabla_v (X_1\wedge X_3) = -\langle X_2, v\rangle f_c(s,\varphi)X_2\wedge X_3, \ \ 
\nabla_v (X_2\wedge X_3) =  \langle X_2, v\rangle f_c(s,\varphi) X_1\wedge X_3.
$$
\end{proposition}
\begin{proof}
These only involve straightforward calculations in Riemannian geometry.
\end{proof}

\subsection{The equations of motion on $\mathcal{N}$}
We write the equations of motion (Equations (\ref{motion_force}) with constant external force) separately  on  the flat and curved parts of $\mathcal{N}$.  On the flat parts, $\mathbb{S}=0$ 
and the system decouples into Newton's equation for the motion of the center of the moving ball and constant $S$ (that is, $S$ is parallel transported along the
trajectory of the center of the ball):
$$\frac{\nabla \dot{x}(t)}{dt}=-ge_3, \ \ \frac{\nabla S}{dt}=0, $$
where `$\cdot$' indicates derivative in $t$.
The covariant derivative of $\dot{x}$ is simply $\ddot{x}$ on $\mathcal{N}^\pm$, so $$x(t)=x(0) +\dot{x}(0)t -\frac{g}2  t^2 e_3,  \ \  S(t)=S(0).$$

Let us focus on the curved part $\mathcal{N}^c$ and rewrite the equations of motion in the orthonormal frame
 $X_1, X_2, X_3$. 
Let $t\mapsto \xi(t)\in \mathcal{M}_{x(t)}$ be a motion of the rolling system, with $\xi(t)=(S(t), v(t))$. (See Section \ref{equations_M_section}.)
The linear and angular velocities are expressed in the orthonormal frame as
$$v(t) =\sum_{i} v_i(t)X_i(x(t)), \ \ S(t)=\frac12 \sum_{i,j} S_{ji}(t)X_i(x(t))\wedge X_j(x(t)).$$
Using the properties of the frame fields summarized in Proposition \ref{calculations_frame}, we obtain the derivatives:
$$\frac{\nabla v}{dt} = \sum_i \dot{v}_i X_i + f_c v_2 (v_2 X_1-v_1 X_2), \ \frac{\nabla S}{dt}=\frac12\sum_{i,j}\dot{S}_{ji} X_i\wedge X_j +v_2 f_c (S_{32}X_1-S_{31}X_2)\wedge X_3.$$

In addition,
$$S\mathbb{S}  v = S_{12}\left( v_2 f_s  X_1 +  \frac{v_1}{r}  X_2 \right)+\left(\frac{v_1}{r}S_{13}-v_2f_sS_{23}\right) X_3$$
and 
$$(\mathbb{S}v)\wedge v=-v_1v_2\left(f_s+\frac1r\right)X_1\wedge X_2  -v_3 \left(\frac1r v_1 X_1-v_2 f_s X_2\right)\wedge X_3.$$
Plugging these expressions into the equations of motion gives the following proposition.

\begin{proposition}\label{equations_4}
The rolling equations, on the curved part of $\mathcal{N}_r$, expressed in terms of the orthonormal frame $X_1, X_2, X_3$ become the following system:
$$\dot{v}_1= -f_c v^2_2 -\eta f_s v_2S_{12}, \ \  \dot{v}_2 = f_c v_1 v_2 - \frac{\eta}r v_1S_{12}, \ \ \dot{S}_{12} = \eta \left(f_s + \frac1r\right) v_1 v_2$$
and 
$$\dot{v}_3 =\eta \left(f_s v_2 S_{23}- \frac1r v_1 S_{13}\right)-g, \ \  \dot{S}_{13}  = -f_c v_2 S_{23} +\frac{\eta}{r} v_1 v_3, \ \ 
\dot{S}_{23} = f_c v_2 S_{13} -{\eta} f_s v_2 v_3.$$
(Note the separation into equations that contain the subindex $3$  and those that do not.)
\end{proposition}

It is worth noting the hierarchical  structure of the equations of Proposition (\ref{equations_4}): the first set of equations involving $v_1, v_2, S_{12}$
are independent of $v_3, S_{13}, S_{23}$ and $x_3$. It represents the motion projected to a cross-section. So the full set of equations can be approached by first solving that first set, then feeding the resulting solutions
into the second set, which now becomes a time-dependent system. We will explore this observation in greater detail in the next section.
For now, further note that the kinetic energy of the system is
$$E=\frac{m}2\left(v_1^2+v_2^2+v_3^2 + S_{12}^2+S_{13}^2+S_{23}^2\right) = E_1+E_2$$
where 
$$E_1 =  \frac{m}2 \left(v_1^2 +v_2^2 + S_{12}^2\right), \ \ E_2 = \frac{m}2\left(v_3^2 + S_{13}^2+S_{23}^2\right)$$
and 
$$\frac{dE_1}{dt}=0, \ \ \frac{dE_2}{dt}=-mg v_3.$$
Thus the transversal motion is conservative by itself, while $E_2+mgx_3$ is conserved, where $mgx_3$ is the gravitational potential energy.
The above proposition should be compared with Theorem \ref{theorem_project} about no-slip billiard systems.
\subsection{Rolling flows on pancake surfaces}
 It follows from Proposition \ref{equations_4} that the motion on $\mathcal{N}_r=\overline{\mathcal{N}}_r$ decomposes into the transversal, free rolling motion on the pancake surface
 $\overline{\mathcal{N}}_r$ (see Figure \ref{Sinai_pancake})    and longitudinal motion governed by the second set of equations in that proposition.  Thus the transversal motion is entirely equivalent to the rolling of a $3$-dimensional ball on the cross-sectional plate $\mathcal{C}$, the $3$-dimensional ball having the
 same moment of inertia parameter $\eta$ as the $4$-dimensional one. 
 
  \begin{figure}[htbp]
\begin{center}
\includegraphics[width=3.0in]{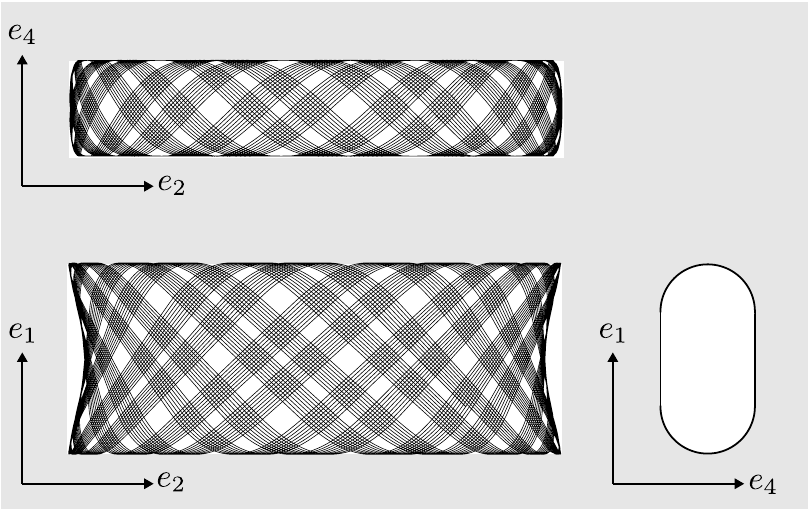}\ \ 
\caption{{\small  Transversal trajectory in $\overline{\mathcal{N}}_r$  of the center of rolling ball on the pancake surface for $P$ bounded by two parallel planes. The cross-section of $\overline{\mathcal{N}}_r$ (contained in the plane spanned by $e_4$ and $e_1$) is stadium-shaped (see bottom-left of Figure \ref{NH_vert_strip}). The longitudinal direction is $e_2$. 
}}
\label{two_planes_pancake}
\end{center}
\end{figure}

 Let us illustrate this idea with the example for which
 $P$ is the region bounded by a pair of parallel vertical planes in $\mathbb{R}^3$. In this case, the pancake surface is essentially the same as the cylinder with
  a stadium-shaped cross-section shown in Figure \ref{NH_strip}, except that the direction perpendicular to $E_1$ and $E_2$ in the figure corresponds to the
 coordinate vector $e_4$ in $\mathbb{R}^4$.  

 \begin{figure}[htbp]
\begin{center}
\includegraphics[width=5.0in]{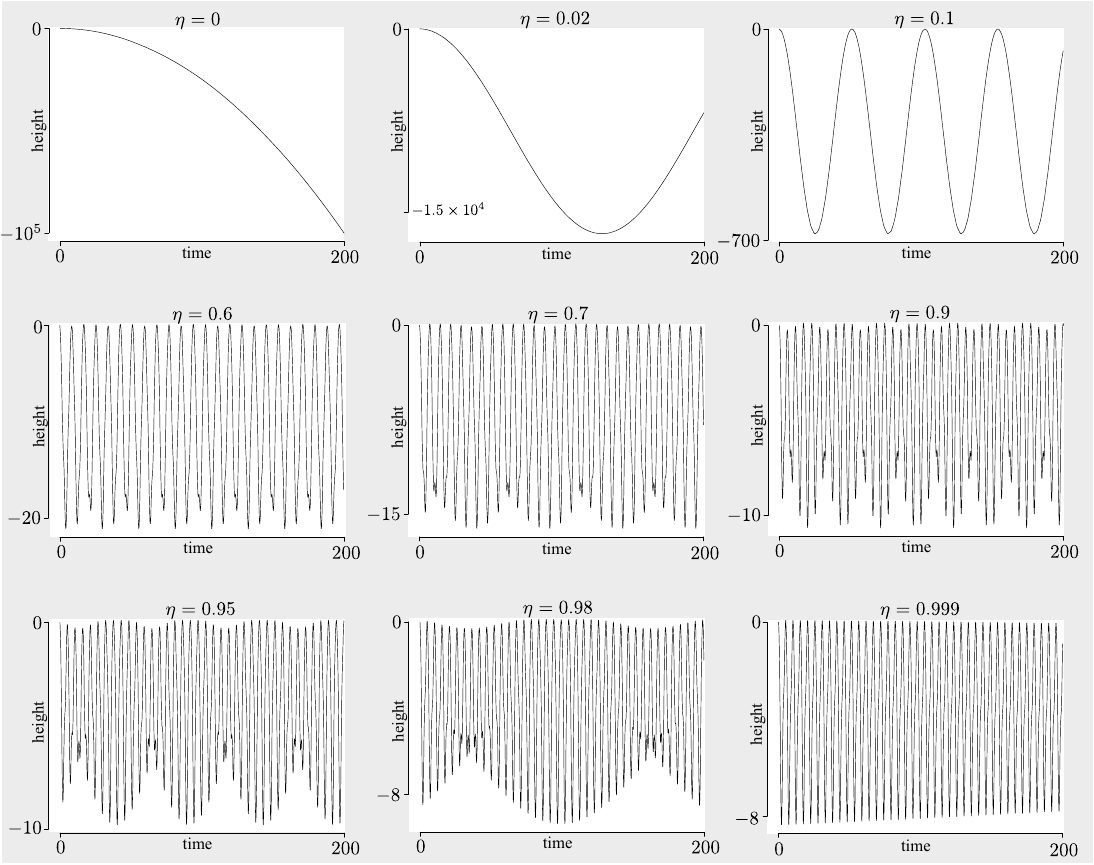}\ \ 
\caption{{\small  Height function of time ($x_3(t)$) for the nonholonomic billiard on $P$ consisting of the region between two vertical planes in $\mathbb{R}^3$
subject to gravity acceleration. When the moment of inertia parameter $\eta$ is zero (top left), the $4$-dimensional rolling ball shows the characteristic parabolic free fall curve. For any positive value of $\eta$, initial free fall eventually rebounds and the motion becomes periodic. (For $\eta=0.999$, the periodicity is not apparent in the time range of the above plots.)
}}
\label{height_eta}
\end{center}
\end{figure}

Before writing the equations of motion for this example, define on $\mathcal{N}_r$ the function
 $$ \zeta(x)=\begin{cases}
 \frac{\eta v_1}{r} & \text{ for } x\in  \mathcal{N}^c_r\\
 0 & \text{ for  } x\in  \mathcal{N}^\pm_r.
 \end{cases}$$
   Note that this function is invariant under translation in the $x_3$-direction, so $\zeta(x)=\zeta(\overline{x})$, where $\overline{x}$ is the projection of $x$ to $\overline{\mathcal{N}}_r$.
   Additionally, this definition anticipates that $v_1$ is a constant of motion.
    Then
 $$ \dot{v}_1 = 0, \ \ \dot{v}_2=-\zeta(x)S_{12},  \ \ \dot{S}_{12}=\zeta(x) v_2$$ 
 are the equations for the transversal motion and 
  $$\dot{v}_3=-\zeta(x)S_{13} - g, \ \ \dot{S}_{13} =\zeta(x)v_3, \ \ \dot{S}_{23}=0$$
  are the equations for the longitudinal motion. 
  These equations are solved just as was done in the example  in Section \ref{vertial_strip_g}.
  Boundedness of the height function ($x_3$), for this example, can be shown analytically as in the case of dimension $3$, by the argument used in \cite{CCCF20}, Proposition 9.

\begin{figure}[htbp]
\begin{center}
\includegraphics[width=4in]{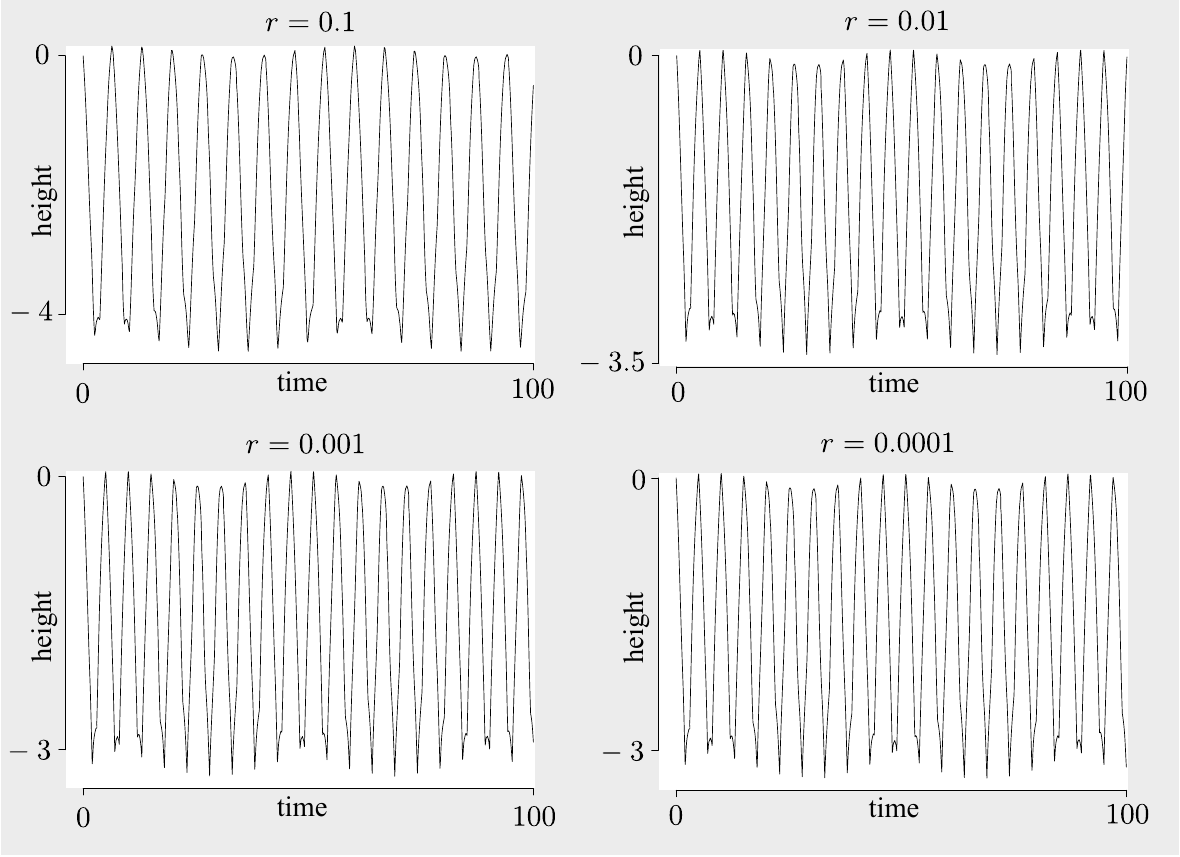}\ \ 
\caption{{\small The nonholonomic billiard system converges to a no-slip billiard system as the radius of the rolling ball tends to $0$.
}}
\label{radius_limit}
\end{center}
\end{figure}

  \begin{theorem}
  The nonholonomic billiard in the region $P$ bounded by two vertical planes in $\mathbb{R}^3$, under constant acceleration due to gravity,
  nonzero initial velocity component $v_1$ (ensuring that the ball will reach the boundary of $P$), and positive moment of inertia parameter, has a periodic (hence bounded) height function.
  \end{theorem}

  Let us see the effect that the moment of inertia parameter has on the height function $x_3(t)$. (Figure \ref{height_eta}.) For the  graphs of Figure \ref{height_eta} we chose
  $L=1, r=0.5$, $g=5$,  $v_1(0)=1$, $v_3(0)=-1$ (ball initially rolling from the right to the left vertical plate), $S_{13}=-0.5$. (These values are mostly arbitrary, but 
  the corresponding  height functions are representative of what is observed generally.)  The moment of inertia parameter $\eta$ must lie in the interval $[0,1)$. 
  For a ball of dimension $n$, if all the mass is uniformly distributed on a thin shell at the surface,  $\eta=\sqrt{2/(2+n)}$ which, in dimension $n=4$, is
  approximately $0.577$. Larger values still make sense   if we imagine   mass  distribution extending  beyond the distance $r$ from
  the center of the ball to the plate $P$ on which it rolls,  as in a yo-yo.

 Figure \ref{radius_limit} shows similar graphs  to those in Figure \ref{height_eta}, for the parameters  $L=1$,  $\eta=0.39$, $g=1$, $v_3(0)=-1$, $v_1(0)=1$, and decreasing  values of $r$. This is numerical evidence that the   height function stabilizes. The limit value as the radius tends to zero is the no-slip billiard limit.

 \subsection{Remarks about the circular cylinder}
 A detailed study of the nonholonomic billiard system on the circular cylinder in dimension $4$, which it is hoped will illuminate the observations made about the no-slip billiard system in the $3$-dimensional cylinder under gravity, will be undertaken in a future study. This is more subtle than the above system consisting of two parallel planes in that boundedness of orbits for the no-slip system is observed for certain initial conditions but not all. Due to the curvature of the boundary curve of 
 $\mathcal{C}$ (see the beginning of Section \ref{section_cylinder} for the definition), the rolling-around-the-edge dynamic cannot be expressed in an explicit analytic form as for the two parallel planes example. In fact, it is not easy to determine a priori for which velocities at a point on the boundary of $\mathcal{N}_r^c$, about to enter into
 this curved part will finally exit it into $\mathcal{N}_r^+$ or $\mathcal{N}_r^-$. (Rolling back into the same flat side from which it entered $\mathcal{N}^c_r$ is similar to
 the  ``friendly roll'' phenomenon of a basketball rolling around the rim  before falling in or out.) Figure \ref{time_distance_portrait} gives some idea of this dynamic.

 \begin{figure}[htbp]
\begin{center}
\includegraphics[width=4.0in]{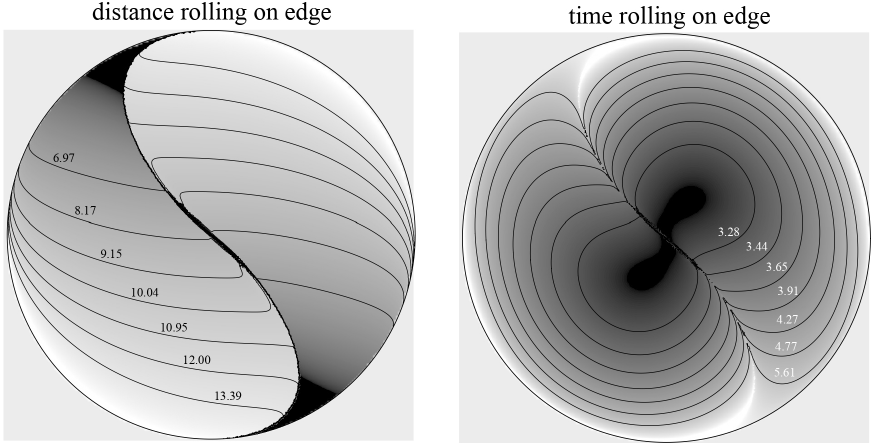}\ \ 
\caption{{\small  Time spent rolling on the circular edge (right) and distance from the initial point along the edge  before reentry in the flat part of $\mathcal{N}_r$, shown
as a function of the initial velocities. The vertical axis indicates the tangential spin variable $s$ (component of scaled angular velocity whose axis is perpendicular
to the $x_3$-slice of $\mathcal{N}_r$) at the moment of crossing from flat to curved part; the horizontal axis indicates the component of the velocity of the center of the ball tangential to the boundary of the flat part. The other component of the center velocity is obtained from this and $s$ from energy conservation. (Here we used $r=R$.)
}}
\label{time_distance_portrait}
\end{center}
\end{figure} 

In order to make sense of these images, recall that the linear and  velocities on an $x_3$-slice of $\mathcal{N}_r$ define a vector in a sphere of radius $\sqrt{2E_1/m}$
where $E_1$ is the constant kinetic energy of the cross-section subsystem. (See the remark following Proposition \ref{equations_4}.) This three-velocity at the moment 
of crossing the separation circle between flat and curved parts of the cross-section of $\mathcal{N}_r$ lies in a hemisphere whose pole points into the curved part.
The two discs of Figure \ref{time_distance_portrait}  are the flattening of this hemisphere. The functions represented by their level curves are, on the right, the time spent 
at the curved part before reentering the flat part and, on the left, the distance around the rim of $\mathcal{C}$ between the places of entry into $\mathcal{N}_r^c$ and exit from it.

  Despite these complications, we can see, numerically, that the nonholonomic billiard system indeed approximates, for small $r$,  the corresponding no-slip system.
 Figure \ref{circular_pancake} shows a
 $2$-dimensional slice of $\mathcal{N}_r$ and a segment of trajectory on it. (The dashed circular line shows the boundary between the flat and curved parts of the surface projected flat on the plane.) When the ball's radius tends to zero, the curve approximates the trajectory of a no-slip billiard system, with the characteristic double caustic.

   \begin{figure}[htbp]
\begin{center}
\includegraphics[width=2.5in]{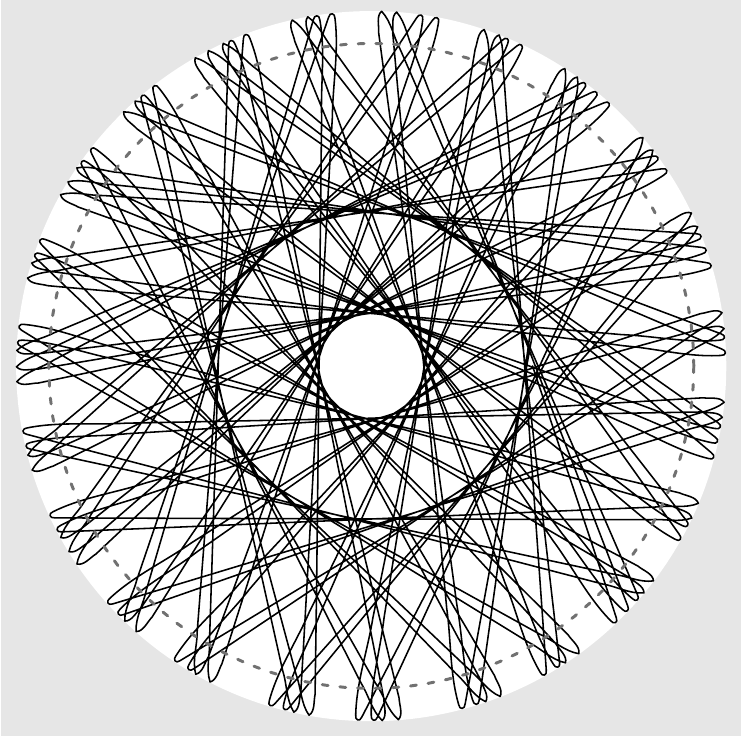}\ \ 
\caption{{\small A trajectory of the nonholonomic billiard ball of radius $r=0.1$ on a solid $3$-dimensional cylinder in $\mathbb{R}^4$, whose $2$-dimensional cross section
is a disc of radius $1$. Shown is 
(the two-dimensional projection of) a  $3$-dimensional $x_3$-slice of $\mathcal{N}_r$ (a circular pancake surface). The characteristic 
double-caustic of the no-slip billiard system in a disc is apparent, but instead of a sharp collision at the boundary we see the smooth rolling-around-the-corner reflection. 
}}
\label{circular_pancake}
\end{center}
\end{figure}

 As for the no-slip billiard system   of Figure \ref{cylinder}, the height function $x_3(t)$ of the nonholonomic billiard on the solid cylinder (with a $4$-dimensional ball)
 will typically accelerate downward in an oscillating fashion, but also   admits for certain initial conditions bounded motion, as Figure 
 \ref{circular_x3} seems to illustrate. (In that figure, the system's parameters are $R=1$, $r=0.1$, $g=1$, $v_{\text{\tiny initial}}=(-0.2, 1, 0, 0)$ for the top example,
$v_{\text{\tiny initial}}=(-2, 1, 0, 0)$ for the bottom example,    $S_{21}=-0.61$, $S_{31}=0$, $S_{32}=1$.) 
The condition of rolling first impact, which we know to be sufficient for bounded trajectories, does not have admit an obvious counterpart for the nonholonomic billiard system. However, the numerical experiments suggest that, in both cases, boundedness of trajectories is implied by the geometric feature that  the two caustic circles typically exhibited by the two-dimensional ($x_1x_2$) projection collapse to a single one.

   \begin{figure}[htbp]
\begin{center}
\includegraphics[width=4.5in]{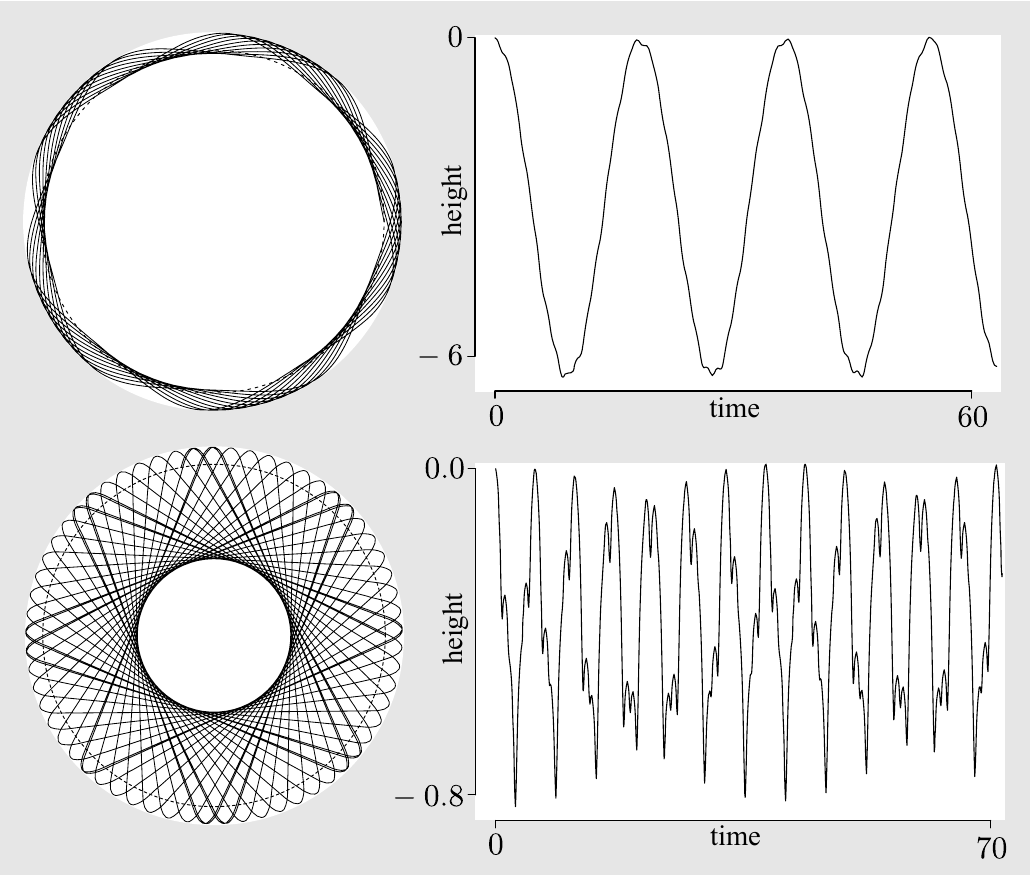}\ \ 
\caption{{\small  Segments of trajectory (top: with short intervals over the flat part of $\mathcal{N}_r$, i.e., skipping motion) for $r=0.1$. Initial conditions were chosen
to produce apparent bounded motion, similar to the behavior observed in Figure \ref{cylinder} under the rolling first impact assumption. However, like
for the no-slip billiard systems in a solid cylinder, the typical initial condition shows oscillation superimposed to  falling. We observe that
bounded motion appears in the numerical examples to coincide with the collapse of the two caustic circles into one.
}}
\label{circular_x3}
\end{center}
\end{figure} 

These remarks and numerical work go a certain distance in validating the proposal of this paper that
 the nonholonomic billiard system can help to illuminate the behavior of the corresponding no-slip billiard system in one dimension lower, but the analytic elaboration of this idea will require further study.

\end{document}